\documentclass{article}

\usepackage[labelfont=bf]{caption}

\usepackage{amsmath,amsfonts,amssymb,amsthm,booktabs,color,epsfig,graphicx,hyperref,url,multirow}

\newtheorem{theorem}{Theorem}
\newtheorem{lemma}{Lemma}

\newtheorem{remark}{Remark}

\newtheorem*{assumption}{Assumption}

\newcommand{\bx}{\boldsymbol{x}}
\newcommand{\by}{\boldsymbol{y}}
\newcommand{\bw}{\boldsymbol{w}}
\newcommand{\bt}{\boldsymbol{t}}
\newcommand{\bs}{\boldsymbol{s}}
\newcommand{\bzero}{\boldsymbol{0}}

\begin{document}

\title{Composite likelihood estimation for a gaussian process under fixed domain asymptotics} 

\author{
{ \normalsize Fran\c{c}ois Bachoc$^{*1}$, Moreno Bevilacqua$^2$, Daira Velandia$^3$ }\\
{ \normalsize $1$ Institut de Math\'ematiques de Toulouse, France.} \\
 { \normalsize $2$ Departamento de Estad\'istica, Universidad de Valpara\'iso } \\
{ \normalsize $3$ Millennium Nucleus  Center for the Discovery of Structures in Complex Data, Chile. } \\
{ \normalsize Departamento de Estad\'istica, Universidad de Valpara\'iso, Chile. } \\
{ \normalsize $*$ Corresponding author. Email address: \url{francois.bachoc@math.univ-toulouse.fr}  }
}

\maketitle

\begin{abstract}
We study the problem of estimating the covariance parameters of a one-dimensional Gaussian process with exponential covariance function under fixed-domain asymptotics.
We show that the weighted pairwise maximum likelihood
estimator of the microergodic parameter can be consistent or inconsistent. This depends on the range of admissible parameter values in the likelihood optimization. On the other hand, the weighted pairwise conditional maximum likelihood
estimator is always consistent. Both estimators are also asymptotically Gaussian when they are consistent. Their asymptotic variances are larger or strictly larger than that of the maximum likelihood estimator.
A simulation study is presented in order to compare the finite sample behavior of the pairwise likelihood estimators with their asymptotic distributions. For more general covariance functions, an additional inconsistency result is provided, for the weighted pairwise maximum likelihood estimator of a variance parameter. \\

\end{abstract}

\section{Introduction\label{sec:1}}

Gaussian processes are widely used in statistics to model spatial data. When fitting a Gaussian field, one has to deal with the issue of estimating its covariance function. In many cases, it is assumed that this function belongs to a given parametric model or family of covariance functions, which turns the problem into a parametric estimation problem. Within this framework, the maximum likelihood estimator (MLE) 
\cite{Ste1999,Rasmussen2006} of the covariance parameters of a Gaussian stochastic process observed in $\mathbb{R}^d$, $d \geq 1$, has been deeply studied in the two following asymptotic frameworks. 

 The fixed-domain asymptotic framework, sometimes called infill asymptotics \cite{Ste1999, Cre1993}, corresponds to the case where more and more data are observed in some fixed bounded sampling domain (usually a region of $\mathbb{R}^d$). The increasing-domain asymptotic framework corresponds to the case where the sampling domain  increases with the number of observed data. 
Furthermore, the distance between any two sampling locations is bounded away from $0$. The asymptotic behavior of the MLE of the covariance parameters can be quite different under these two frameworks \cite{ZhaZim2005}. 

 Under  increasing-domain asymptotics, for all covariance parameters, the MLE is consistent and asymptotically normal under some mild regularity conditions. The asymptotic covariance matrix is equal to the inverse of the Fisher information matrix \cite{MarMar1984,shaby12tapered,Bac2014}.
 
  The situation is significantly different under fixed-domain asymptotics. Indeed, two types of covariance parameters can be distinguished: microergodic and non-microergodic parameters. A covariance parameter is microergodic if two Gaussian measures constructed from different values of it are orthogonal, see \cite{IbrRoz1978,Ste1999}. It is non-microergodic if two Gaussian measures constructed from different values of it are equivalent. Non-microergodic parameters cannot be estimated consistently, but misspecifying them asymptotically results in the same statistical inference as specifying them correctly \cite{AEPRFMCF,BELPUICF,UAOLPRFUISOS,Zha2004}. 
In the case of isotropic Mat\'ern covariance functions with $d \leq 3$, \cite{Zha2004} shows that only a reparametrized quantity obtained from the scale and variance parameters is microergodic. The asymptotic normality of the MLE of this microergodic parameter is then obtained in \cite{ShaKau2013}.
Similar results for the special case of the exponential covariance function were obtained previously in \cite{Yin1991}. 

The maximum likelihood method is generally considered the best option for estimating the covariance parameters of a Gaussian process (at least in the framework of the present paper, where the true covariance function does belong to the parametric model, see also \cite{Bachoc13cross,bachoc2018asymptotic}). Nevertheless, the evaluation of the likelihood function under the Gaussian assumption requires to solve a system of linear equations and to compute a determinant. For a data set of $n$ observations, the computational burden is $O(n^3)$. This makes maximum likelihood computationally impractical for large data sets. This fact motivates the search for estimation methods with a good balance between computational complexity and statistical efficiency. Among these methods, we can mention low rank approximation (see \cite{stein14limitations} and the references therein for a review), sparse approximation \cite{hensman2013}, covariance tapering \cite{furrer2006covariance,kaufman08covariance}, Gaussian Markov random fields approximation \cite{rue05gaussian,datta16hierarchical}, submodel aggregation \cite{hinton2002training,trespBCM,caoGPOE,deisenroth2015,van2015optimally,rulliere2018nested} and composite likelihood (CL).

CL indicates a general class of objective functions based on the likelihood of marginal or conditional events \cite{Varin:Reid:Firth:2011}. This kind of estimation method has two important benefits. First, it is generally appealing when dealing with large data sets. Second, it can be helpful when it is difficult to specify the full likelihood.
In this paper we focus on a specific type of CL called pairwise likelihood. Examples of the use of pairwise likelihood  in the Gaussian and non Gaussian case can be found in
\cite{Bevilacqua:Gaetan:2015,Bevilacqua:Gaetan:Mateu:Porcu:2012,Feng_et_al:2014,Guan:2006,Heagerty:Lele:1998} to name a few. 

The increasing-domain asymptotic properties of various weighted pairwise likelihood estimators in the Gaussian case 
can be found in \cite{Bevilacqua:Gaetan:2015}.
Under this setting, for all covariance parameters, these pairwise likelihood estimators are consistent and asymptotically normal under some mild regularity conditions. 

There are no results under fixed-domain asymptotics for CL to the best of our knowledge.
In this paper we study the fixed-domain asymptotic properties of the  pairwise likelihood estimation method. A one-dimensional Gaussian process with exponential covariance model is considered. This covariance model is particularly amenable to theoretical analysis of the MLE. Indeed, its Markovian
property yields an explicit (matrix-free) expression of the likelihood function \cite{Yin1991}. Thus, the MLE for this model has been studied in  \cite{CheSimYin2000,chang2017mixed,Yin1991}, and also in higher dimension in \cite{ying93maximum,vdV1996,AbtWel1998}. 

Under the parametric model given by the one-dimensional exponential covariance function, the microergodic parameter is the product of the variance and the scale (see \cite{Yin1991} and Section \ref{sec:estimators}). 
Under the same setting as in \cite{Yin1991}, we study the consistency and asymptotic normality of the weighted pairwise maximum likelihood estimator (WPMLE) and of the weighted pairwise conditional maximum likelihood estimator (WPCMLE) of the microergodic parameter. Our results show that the WPMLE is inconsistent if the range of admissible values for the variance parameter is too large compared to that for the scale parameter (Theorem \ref{thm:consistency:inconsistency:pl}, \rm(i) and \rm(iii)). Conversely, they show that the WPMLE is consistent if the range of admissible values for the variance parameter is restricted enough compared to that for the scale parameter (Theorem \ref{thm:consistency:inconsistency:pl}, (ii) and (iv)). In contrasts, we prove that the WPCMLE is consistent regardless of the ranges of admissible values for the scale and variance (Theorem \ref{thm:consistency:inconsistency:plc}). 
Both estimators are also asymptotically Gaussian in the cases where they are consistent (Theorem \ref{thm:normality}). The asymptotic variance of these estimators is no smaller than that of the MLE. Furthermore, it is strictly larger for equispaced observation points, when non-zero weights are given to pairs of non successive observation points. 

The proofs of the asymptotic results rely on Lemma \ref{lemma:for:PL:CL:to:lik}. This lemma expresses the  weighted pairwise (conditional) likelihood criteria as combinations of full likelihood criteria for subsets of the observations. This enables to exploit the asymptotic approximations obtained by \cite{Yin1991} for the full likelihood.
The consistency results are then obtained similarly as in \cite{Yin1991}.
The inconsistency results are obtained by a careful analysis of the limit weighted likelihood criterion, together with results on extrema of Gaussian processes (see the proof of Theorem \ref{thm:consistency:inconsistency:pl}).
Finally, the asymptotic normality results are obtained by means of a central limit theorem for weakly dependent variables, together with the use of Lemma \ref{lemma:for:PL:CL:to:lik} and of the approximations in \cite{Yin1991}.

It turns out that a suitable choice of the  symmetric weights used in the pairwise likelihood
is crucial for improving the asymptotic statistical efficiency   of the method. 
In particular, the choice of compactly supported weights 
with a specific  cut-off generally decreases the asymptotic variance.
As a special case, 
if the    pairwise
likelihood is constructed only from adjacent pairs, then the  asymptotic efficiency of the MLE is attained. 

These results are consistent with the  increasing-domain setting. Indeed, it has been shown, via simulation, that
weighting schemes  downweighting observations that are far apart in
time and/or in space,
are  preferable to the  pairwise
likelihood involving all possible pairs
\cite{VV2006,JJ2009,Bevilacqua:Gaetan:2015}.
A simulation study compares the finite sample properties of  the weighted pairwise likelihood estimators
with their  asymptotic distributions,
using the MLE as a benchmark.
Finally, we provide an additional inconsistency result for the WPMLE of a variance parameter, when the correlation function is known and allowed to be significantly more general than the exponential one.

The paper falls into the following parts. In Section \ref{sec:estimators}, pairwise (conditional) likelihood estimation is introduced.
In Section \ref{sec:cons}, the consistency and inconsistency results are provided. In Section \ref{sec:AN}, the asymptotic normality results are provided, together with the analysis of the asymptotic variance.
In Section \ref{sec:exp}, the numerical results are discussed.
Concluding remarks are given in Section \ref{sec:conc}. The proofs are postponed to \ref{appendix:proofs}. The more general inconsistency result of the WPMLE is provided in \ref{appendix:inconsistency}.

\section{Pairwise (conditional) likelihood estimation \label{sec:estimators}}

In the following, we will consider a stationary zero-mean Gaussian process $Z=\{ Z(s), s \in  [0,1]  \}$,
with covariance function given by, for $s,t \in \mathbb{R}$,  
\begin{equation}\label{exp:cov}
{\rm cov}\{Z(s),Z(t)\}=\sigma_0^2 e^{-\theta_0 |s-t|}
\end{equation}
for fixed $\theta_0 , \sigma_0^2 \in (0,\infty)$.

For $\theta , \sigma^2 \in (0,\infty)$, let $\psi = (\theta , \sigma^2)^{\top}$, call $\sigma^2$ the variance parameter and call $\theta$ the scale (inverse of correlation decay) parameter. Let $\psi_0 = (\theta_0 , \sigma_0^2)^{\top}$. Throughout this paper, $\theta$ and $\sigma^2$ denote candidate covariance parameters, at which likelihood or pairwise likelihood criteria can be evaluated (see \eqref{eq:pairwise:lik:expression} and \eqref{eq:pairwise:cond:lik:expression} below), while $ \theta_0$ and $\sigma_0^2$ are the true covariance parameters, that provide the true covariance function of $Z$ through \eqref{exp:cov}.

The Gaussian process $Z$ is  known as the stationary Ornstein-Uhlenbeck process. Let $\big( s_i^{(n)} \big)_{n \in \mathbb{N}, i \in \{1,\ldots,n\}}$ be a triangular array of observations points in $[0,1]$. For $n \in \mathbb{N}$, let $\{s_1,\ldots,s_n\} = \{s_1^{(n)},\ldots,s_n^{(n)}\}$ for simplicity and assume without loss of generality that $s_1 < \ldots < s_n$.
The observation vector can thus be written as  $Z_n=(Z(s_1),\ldots,  Z(s_n))^{\top}$.

For $s,t \in [0,1]$, we let 
\[
\ell_{ s,t }(\psi)
= 2 \ln (\sigma^{2})+ \ln (1-e^{-2\theta |s-t|})+
\frac{1}{\sigma^2} Z(s)^2
+
\frac{ \{ Z(t)- e^{-\theta | s-t | } Z(s)\}^2 }{\sigma^{2}(1-e^{-2\theta |s-t|})}  
\]
be the log-likelihood criterion associated to the bivariate random vector  $(Z(s),Z(t))^\top$. The quantity $\ell_{ s,t }(\psi)$ is $-2$ times the logarithm of the probability density function of the vector  $(Z(s),Z(t))^\top$, when considering that $Z$ has covariance function given by \eqref{exp:cov} with $\psi_0$ replaced by $\psi$.

We also let
\[
\ell_{ t | s }(\psi)
= \ln (\sigma^{2})+ \ln (1-e^{-2\theta |s-t|})
+
\frac{ \{ Z(t)- e^{-\theta | s-t | } Z(s) \}^2 }{\sigma^{2}(1-e^{-2\theta |s-t|})} 
\]
be the conditional log-likelihood criterion of $Z(t)$ given $Z(s)$. 
The quantity $\ell_{ t|s }(\psi)$ is $-2$ times the logarithm of the conditional probability density function of $Z(t)$  given $Z(s)$, when considering that $Z$ has covariance function given by \eqref{exp:cov} with $\psi_0$ replaced by $\psi$. 

We let $(w_{i,j}^{(n)})_{n \in \mathbb{N}, i,j \in \{1,\ldots,n \},i \neq j}$ be a triangular array of weights in $[0,\infty)$. We let $w_{i,j}^{(n)} = w_{i,j}$ for simplicity.
Then, define the weighted pairwise log-likelihood function ${p\ell}_n$ as
\begin{equation} \label{eq:pairwise:lik:expression}
{p \ell}_{n}(\psi) =  
\sum_{i=1}^{n-1} \sum_{j=i+1}^n
w_{i,j} \ell_{ s_i,s_j }(\psi)
\end{equation}
and the weighted pairwise conditional log-likelihood function ${pc\ell}_n$ as
\begin{equation}  \label{eq:pairwise:cond:lik:expression}
{pc\ell}_{n}(\psi) =  
\sum_{i=1}^{n} \sum_{\substack{j =1 \\ j \neq i}}^n
w_{i,j} \ell_{ s_j | s_i }(\psi).
\end{equation}
Finally, for $J \subset (0 , \infty)^2$, define the WPMLE  as
\begin{equation} \label{eq:def:wpmle}
\widehat{\psi}_{{p \ell}}=\underset{{\psi} \in J}{ \mathrm{argmin}}\,\,{p \ell}_{n} (\psi)
\end{equation}
and the WPCMLE as 
\begin{equation} \label{eq:def:wpcmle}
 \widehat{\psi}_{{pc \ell}}=\underset{{\psi} \in J}{\mathrm{argmin}}\,\,{pc\ell}_{n} (\psi).
\end{equation}

We remark that the WPMLE and WPCMLE are two distinct estimators, which will be shown in Section \ref{sec:cons} to have significantly different behaviors under fixed-domain asymptotics.

In the rest of the paper, we will consider the case where the weights $(w_{i,j}^{(n)})$ are so that for any $n \in \mathbb{N}$ and $i,j \in \{1,\ldots,n\}, i \neq j$, $w_{i,j} = w_{|i-j|}$ with $(w_{k})_{k \in \mathbb{N}}$ a fixed sequence of non-negative numbers, not all of which being zero. We let $K = \rm max \{ i \in \mathbb{N} ; w_i \neq 0 \}$, where we allow for $K= + \infty$ if $(w_k)_{k \in \mathbb{N}}$ has infinitely many non-zero elements. The quantity $K$ is fixed independently of $n$ and $K < + \infty$ if an only if the sequence $(w_k)_{k \in \mathbb{N}}$ has finitely many non-zero elements. We remark that $K \geq 1$ since $(w_k)_{k \in \mathbb{N}}$ is not the zero sequence.

In the following lemma, convenient expressions are given for the weighted pairwise log-likelihood and weighted pairwise conditional log-likelihood functions. 

\begin{lemma} \label{lemma:for:PL:CL:to:lik}
	
	Let, for $k\in \{1,\ldots,{n-1}\}$ and $a \in \{0,\ldots,k-1\}$, $x^{(k,a)}_{j} = s_{1+a+jk}$, for $j \in \mathbb{N}$ so that $1+a+jk \leq n$. Then, with $\lfloor x \rfloor$ the largest integer that is smaller or equal to $x$, and with the convention that $\sum_{i=a}^b r_i = 0$ if $b < a$,
	\begin{equation} \label{eq:pairwise:lik:to:lik}
	{p\ell}_{n}(\psi) =  
	\sum_{k=1}^{n-1}
	w_{k}
	\sum_{a=0}^{k-1}
	\sum_{j=0}^{ \lfloor \frac{n-1-a-k}{  k} \rfloor }
	\ell_{ x^{(k,a)}_{j} ,x^{(k,a)}_{j+1} }(\psi)
	\end{equation}
	and
	\begin{equation} \label{eq:pairwise:cond:lik:to:lik}
	{pc \ell}_{n}(\psi) =  
	\sum_{k=1}^{n-1}
	w_{k}
	\sum_{a=0}^{k-1}
	\sum_{j=0}^{ \lfloor \frac{n-1-a-k}{  k} \rfloor }
	\left\{
	\ell_{ x^{(k,a)}_{j+1} |x^{(k,a)}_{j} }(\psi)
	+ 
	\ell_{ x^{(k,a)}_{j} |x^{(k,a)}_{j+1} }(\psi)	 
	\right\}.
	\end{equation}
	
\end{lemma}

The benefit of Lemma \ref{lemma:for:PL:CL:to:lik} is that the rightmost sum in \eqref{eq:pairwise:cond:lik:to:lik} can be expressed as a (full) log likelihood for a subset (that depends on $k$ and $a$ in \eqref{eq:pairwise:cond:lik:to:lik}) of $\{s_1,\ldots,s_n\}$ (see Lemma 1 in \cite{Yin1991}). Similarly, the rightmost sum in \eqref{eq:pairwise:lik:to:lik} can be expressed conveniently as a function of this (full) log likelihood. We refer to the proofs of Theorems \ref{thm:consistency:inconsistency:pl}, \ref{thm:consistency:inconsistency:plc} and \ref{thm:normality} for details.

\section{Consistency and inconsistency \label{sec:cons}}

It is well-known that, in the case of the exponential covariance model given in \eqref{exp:cov}, the parameters $\sigma^2$ and $\theta$ are non-microergodic and only the parameter $\sigma^2 \theta$ is microergodic \cite{Yin1991,Zha2004}. Hence, we study the consistency, inconsistency and asymptotic normality for the estimators $\hat{\theta}_{{p\ell}}\hat{\sigma}^{2}_{{p \ell}}$ and $\hat{\theta}_{{pc\ell}}\hat{\sigma}^{2}_{{pc\ell}}$ of $\sigma_0^2 \theta_0$ given in \eqref{eq:def:wpmle} and \eqref{eq:def:wpcmle}.

In the next theorem, let $J$ be of the form $[a,b] \times [c,d]$, $[a,b] \times (0 , \infty)$ or $(0 , \infty) \times [c,d]$. The next theorem characterizes, as a function of $J$, the cases where the WPMLE is consistent and those where it is inconsistent.

\begin{theorem} \label{thm:consistency:inconsistency:pl}
	Assume that $K < + \infty$, that $J  \subset (0,\infty)^2 $ and that there exists $(\tilde{\theta},\tilde{\sigma}^2)$ in $J$ so that $\tilde{\theta} \tilde{\sigma}^2 = \theta_0 \sigma_0^2$.
	Then, for fixed $0 < a \leq b < + \infty$ and $0 < c \leq d < + \infty$,
	\begin{itemize}
		\item[{\rm(i)}] If $J = [a,b] \times [ c , d ] $ with $ad > \theta_0 \sigma_0 ^2$ or $ bc < \theta_0 \sigma_0^2 $, then 
		$\hat{\theta}_{{p\ell}} \hat{\sigma}^{2}_{{p \ell}}$ does not converge in probability to $\theta_0 \sigma_0^2$;
		\item[{\rm(ii)}]
		If $J = [a,b] \times [ c , d ] $ with $ad \leq \theta_0 \sigma_0 ^2$ and $ bc \geq \theta_0 \sigma_0^2 $, then 
		$\hat{\theta}_{{p\ell}} \hat{\sigma}^{2}_{{p\ell}}$ converges  to $\theta_0 \sigma_0^2$ almost surely;
		\item[{\rm(iii)}]
		If $J = [a,b] \times (0 , \infty)$ then 
		$\hat{\theta}_{{p\ell}} \hat{\sigma}^{2}_{{p\ell}}$ does not converge in probability to $\theta_0 \sigma_0^2$;
		\item[{\rm(iv)}]
		If $J = (0 , \infty) \times  [c,d]$ then $\hat{\theta}_{{p\ell}} \hat{\sigma}^{2}_{{p\ell}}$ converges  to $\theta_0 \sigma_0^2$ almost surely.
	\end{itemize} 
\end{theorem}

\begin{remark} \normalfont
In Theorem \ref{thm:consistency:inconsistency:pl}, the parameter set $J$ always contains a pair $(\tilde{\theta},\tilde{\sigma}^2)$ so that $\tilde{\theta} \tilde{\sigma}^2 = \theta_0 \sigma_0^2$, where we recall that $\theta_0 \sigma_0^2$ is the unknown microergodic parameter. Thus, in the framework of Theorem \ref{thm:consistency:inconsistency:pl}, the parameter $\theta_0 \sigma_0^2$ is always consistently estimable, and indeed Theorem \ref{thm:consistency:inconsistency:plc} below shows that there exists an estimator, namely the WPCMLE, that is consistent. This does not mean that all the estimators of $\theta_0 \sigma_0^2$ are consistent, and indeed Theorem \ref{thm:consistency:inconsistency:pl} shows that the WPMLE can be inconsistent.
\end{remark}

Theorem \ref{thm:consistency:inconsistency:pl} can be interpreted as follows: The pairwise likelihood criterion can be decomposed as
\[
{p \ell}_{n}(\psi)
= 
S_1( \sigma^2 ) + S_2( \theta , \sigma^2 ),
\]
where $S_1( \sigma^2 )$ is the first triple sum in \eqref{eq:reindexing:pl} in the proof of Theorem \ref{thm:consistency:inconsistency:pl} and $S_2( \theta , \sigma^2 )$ is the second one.
The minimizer of the random function $S_2$ of $\theta$ and $\sigma^2$ would yield a consistent estimator of $\theta_0 \sigma_0^2$. 
However, the random function $S_1$ has random fluctuations that do not become negligible as $n \to \infty$. The consequence of these fluctuations is that the WPMLE of $\theta_0 \sigma_0^2$ is inconsistent if the range of possible values for $\sigma^2$ is too large compared to that for $\theta$ (cases (i) and (iii)). On the contrary, if the range of possible values for $\theta$ is large enough compared to that for $\sigma^2$, then the WPMLE $\hat{\theta}_{p \ell}$ is able to compensate, so to speak, the non-negligible fluctuations of $S_1$. More precisely, independently of the value of $\hat{\sigma}_{p \ell}^2$, in the cases (ii) and (iv), because of the good behavior of $S_2$, the value of $\hat{\theta}_{p\ell}$ will be so that $\hat{\theta}_{p\ell} \hat{\sigma}_{p\ell}^2 $ does converge to $\theta_0 \sigma_0^2$. This discussion is the basis of the proof of Theorem \ref{thm:consistency:inconsistency:pl}, to which we refer for more details. We remark that this discussion applies to the WPMLE, and not to the WPCMLE which is a different estimator and is always consistent, see Theorem \ref{thm:consistency:inconsistency:plc} below.

It should be mentioned that, under increasing-domain asymptotics, the WPMLE of $(\theta_0, \sigma_0^2)$, and thus of $\theta_0 \sigma_0^2$, is consistent, independently of the form of $J$, provided $J$ contains $(\theta_0, \sigma_0^2)$ (see the general results given in \cite{Bevilacqua:Gaetan:2015}). It can indeed be checked that $S_1(\sigma^2)$ has random fluctuations that are typically asymptotically negligible under increasing-domain asymptotics. Hence, we have an additional illustration of the important qualitative differences that can appear when comparing fixed and increasing-domain asymptotics (see also \cite{ZhaZim2005}).

From a practical standpoint, Theorem \ref{thm:consistency:inconsistency:pl} provides a warning when using the WPMLE. In particular, a situation that often occurs in practice is the case (iii), where $\theta$ has a compact range of allowed values and where for any fixed $\theta$, the value of $\sigma^2$ that optimizes the pairwise likelihood (without restriction) can be computed explicitly. This case (iii) is a case of inconsistent estimation of the microergodic parameter. Furthermore, in \ref{appendix:inconsistency}, we provide another inconsistency result for the WPMLE, for more general covariance functions and when a single variance parameter is estimated.

The following theorem gives the consistency of the WPCMLE for the microergodic parameter, for all the cases for $J$ under investigation in Theorem \ref{thm:consistency:inconsistency:pl}. Hence, Theorems \ref{thm:consistency:inconsistency:pl} and \ref{thm:consistency:inconsistency:plc} jointly provide an incentive to use pairwise conditional likelihood rather than pairwise (unconditional) likelihood in practice.

\begin{theorem} \label{thm:consistency:inconsistency:plc}
	Assume that $K < + \infty$.
	Let $0 < a \leq b < + \infty$ and $0 < c \leq d < + \infty$ be fixed. Let $J = [a,b] \times [ c , d ] $ or $J = [a,b] \times (0 , \infty)$ or $J = (0 , \infty) \times  [c,d]$ and assume that there exist $\tilde{\theta},\tilde{\sigma}^2$ in $J$ so that $\tilde{\theta} \tilde{\sigma}^2 = \theta_0 \sigma_0^2$.
	Then, $\hat{\theta}_{{pc\ell}} \hat{\sigma}^{2}_{{pc\ell}}$ converges  to $\theta_0 \sigma_0^2$ almost surely. 
\end{theorem}

\section{Asymptotic normality \label{sec:AN}}

Throughout this section, we assume that $\rm max_{i\in \{1,\ldots,n-1\}} (s_{i+1} - s_i) \to 0$ as $n \to \infty$.
In order to express the asymptotic variance, in the asymptotic normality result in Theorem \ref{thm:normality} below, let us define, for $i \in \mathbb{N}$ and $k \in \mathbb{N}$, with $ i \geq 1$ and $ i+k \leq n$,
\[
W^{2}_{i,i+k}=\frac{\left\{ Z(s_{i+k})-e^{-\theta_{0}(s_{i+k} - s_i)}Z(s_{i})\right\}^{2}}{ \sigma_{0}^{2} (1-e^{-2\theta_{0} (s_{i+k} - s_i)}) }.
\]
Let us also define
\begin{equation}\label{eqtau}
\tau_{n}^{2}=\frac{1}{n}{\rm var}\left\{ \sum_{i=1}^{n-1}\sum_{k=1}^{{\rm min}\{K,n-i\}} w_{k} \left(W^{2}_{i,i+k}-1\right)\right\}.
\end{equation}

Lemma \ref{lemma:approx:tauKn} provides an asymptotic approximation of $\tau_{n}^{2}$ when $K< + \infty$. 

\begin{lemma} \label{lemma:approx:tauKn}
	Assume that $K < + \infty$. 
	Let, for $1 \leq i<j\leq n$ and $1 \leq k<\ell \leq n$, $b_{i,j,k,\ell}$ be defined by $b_{i,j,k,\ell} = b_{k,\ell,i,j}$ and, for $i \leq k$,
	\begin{equation}  \label{eq:bijkl}
	b_{i,j,k,\ell}
	=
	\begin{cases}
	0 & \mbox{if     } j \leq k \\
	\frac{(s_j - s_k)^2}{(s_j - s_i)(s_\ell - s_k)} & \mbox{if     }  k \leq j \leq \ell \\
	\frac{s_\ell - s_k}{s_j - s_i} & \mbox{if     }  k \leq \ell \leq j.
	\end{cases}
	\end{equation}
	
	We have as $n \to \infty$
	\begin{align} \label{eq:tauKn:equivalent}
	\tau_{n}^2
	= 
	\left(
	\frac{2}{n}
	\sum_{i=1}^{n-1} \sum_{j=i+1}^{{\rm min}(n,i+K)}
	\sum_{k=1}^{n-1} \sum_{\ell=k+1}^{{\rm min}(n,k+K)}
	w_{j-i} w_{\ell-k}
	b_{i,j,k,\ell}
	\right)  
	+ o(1).
	\end{align}
\end{lemma}

In Lemmas \ref{lem:larger:variance} and \ref{lem:smaller:variance}, we show that $\tau_{n}^{2}$ is lower and upper bounded as $n \to \infty$.

\begin{lemma} \label{lem:larger:variance}
	Assume that $K < + \infty$.
	Then $\mathrm{\lim \rm inf}_{n \to \infty} \tau_{n}^2 \geq 2 (\sum_{k=1}^K w_k)^2$. 
\end{lemma}

\begin{lemma} \label{lem:smaller:variance}
	Assume that $K < + \infty$.
	Then $\mathrm{\lim \rm sup}_{n \to \infty} \tau_{n}^2 < + \infty$. 
\end{lemma}

Then, the following theorem establishes the asymptotic normality of the WPMLE (in the cases where it is consistent) and of the WPCMLE (in all the cases) of the microergodic parameter. 

\begin{theorem} \label{thm:normality}
	Under the cases (ii) and (iv) of Theorem \ref{thm:consistency:inconsistency:pl}, and under the same conditions as in this theorem, 
	\begin{equation}
	\frac{\sqrt{n}  (\sum_{k=1}^K w_k) }{\tau_{n}\sigma^{2}_{0}\theta_{0}}(\hat{\sigma}_{{p\ell}}^{2}\hat{\theta}_{{p \ell}}-\sigma^{2}_{0}\theta_{0})\stackrel{\mathcal{D}}{\longrightarrow}
	\mathcal{N}
	\left(0, 1\right).\label{dotto}
	\end{equation}
	Furthermore, under the same conditions as in Theorem \ref{thm:consistency:inconsistency:plc}, 
	\begin{equation}
	\frac{\sqrt{n}  (\sum_{k=1}^K w_k) }{\tau_{n}\sigma^{2}_{0}\theta_{0}}(\hat{\sigma}_{{pc\ell}}^{2}\hat{\theta}_{{pc\ell}}-\sigma^{2}_{0}\theta_{0})\stackrel{\mathcal{D}}{\longrightarrow}
	\mathcal{N}
	\left(0, 1\right).\label{dotto:pcl}
	\end{equation}
\end{theorem}

Because of Lemmas \ref{lem:larger:variance} and \ref{lem:smaller:variance}, the rate of convergence is $n^{1/2}$ in Theorem \ref{thm:normality}. Furthermore, it is shown in \cite{{Yin1991}} that the asymptotic variance of the MLE of $\theta_0 \sigma_0^2$ is $2 (\theta_0 \sigma_0^2)^2$. Hence, from Lemma \ref{lem:larger:variance}, the WPMLE and WPCMLE have a larger asymptotic variance than the MLE, which is in agreement with the fact that the MLE is generally considered to be the most efficient statistically. Also, the asymptotic variance of the WPMLE and WPCMLE depends on the triangular array of observation points (see Lemma \ref{lemma:approx:tauKn}), while that of the MLE does not. A similar conclusion was obtained in \cite{bachoc2017cross} when considering a cross validation estimator in the same context.

\begin{remark} \normalfont
It can be shown, from Lemma \ref{lemma:for:PL:CL:to:lik} and from Lemma 1 in \cite{Yin1991}, that the WPMLE and WPCMLE asymptotically coincide with the MLE when $K=1$ (that is when $w_1 >0$ and $w_k=0$ for $k \geq 2$) and under the conditions of Theorem \ref{thm:normality}.  When $K=1$, it is easily seen from \eqref{eqtau} that $\tau_{n}^2 = 2 w_1^2+o(1)$ so that, indeed, the WPMLE and WPCMLE have the same asymptotic variance as the MLE. 
\end{remark}

The following lemma shows that for equispaced observation points, the asymptotic variance of the WPMLE and WPCMLE is strictly larger than that of the MLE when one of the $w_k$, $k \geq 2$ is non-zero. 

\begin{lemma} \label{lem:strictly:larger:variance}
	If $s_i = i/n$ for all $n \in \mathbb{N}$ and $i \in \{1,\ldots,n\}$ and if $2 \leq K < + \infty$, then $\mathrm{\lim \rm inf}_{n \to \infty} \tau_{n}^2 > 2 \left(\sum_{i=1}^K w_k\right)^2$.
\end{lemma}
\begin{remark} \normalfont
The sequence of weights $\{w_1,0,\ldots\}$ with $w_1 >0$ provides the smallest asymptotic variance because of the fact that, in this case, the log-likelihood function is asymptotically equal to a sum of pairwise conditional log likelihoods for consecutive observation points (see Lemma 1 in \cite{Yin1991}). This is specific to the exponential covariance function in \eqref{exp:cov}. For more general families of covariance function, it is possible that it is more efficient to use other weight configurations than $\{w_1,0,\ldots\}$.
\end{remark}

\begin{remark} \normalfont
The log likelihood criterion can be evaluated with a $O(n)$ computational cost (see Lemma 1 in \cite{Yin1991}). Hence, here the WPMLE, WPCMLE and MLE have the same computational cost but in general, the likelihood evaluation cost is $O(n^3)$ in time and $O(n^2)$ in storage, while this cost remains $O(n)$ for weighted pairwise likelihood criteria (when $K < + \infty$)
irrespectively of the covariance function.
\end{remark}

\begin{remark} \normalfont
In Sections \ref{sec:cons} and \ref{sec:AN}, we have provided consistency and asymptotic normality results under fixed-domain asymptotics. As discussed after Theorem \ref{thm:consistency:inconsistency:pl}, some of the conclusions obtained differ from those arising from an increasing-domain asymptotic analysis. Hence, given a finite set of observation points and observed values, the question of whether to consider increasing-domain or fixed-domain asymptotics as most relevant is important. 

Let us discuss this question in the general setting of a stationary isotropic Gaussian random field on $\mathbb{R}^d$. Let us consider the correlation distance $\Delta \geq 0$, such that the correlation for two points at distance more than $\Delta$ is expected to be less than a given threshold, for instance $0.05$. Let us also consider that there are $n$ observation points $\bs_1,\ldots,\bs_n$, and consider the diameter $ D =  {\rm max}_{i,j \in \{1,\ldots,n \}} ||\bs_i - \bs_j||$ of the point cloud. Then, a first rule of thumb is to consider that fixed-domain asymptotics is relevant when $n$ is large compared to $(D / \Delta)^d$. Conversely, increasing-domain asymptotics is relevant if $n$ is comparable to $(D / \Delta)^d$. Hence, fixed-domain asymptotics is more relevant when $D$ decreases and $\Delta$ increases and conversely for increasing-domain asymptotics. We also remark that when $n$ is large compared to $(D / \Delta)^d$ and $D / \Delta$ is large, that is the observation domain is large and the observation points are dense in this domain, then a third asymptotic setting called mixed-domain asymptotics \cite{chang2017mixed} can be most relevant, although less theory is available in this case.

Beyond this rule of thumb, the reference \cite{ZhaZim2005} is, to the best of our knowledge, the only one that provides a quantitative comparison of the approximations of finite sample situations, brought by fixed-domain and increasing-domain asymptotics. The authors focus on maximum likelihood estimation of covariance parameters, and consider settings where two asymptotic distributions of the estimators can be obtained from both asymptotic frameworks. In simulations, they show that the fixed-domain asymptotic distributions are closer to the finite sample distributions than the increasing-domain ones, especially for the microergodic covariance parameters.

We can also mention other specific settings where fixed-domain asymptotics is more relevant. When studying the properties of sequential procedures, like in optimization \cite{bect2019supermartingale,vazquez2010convergence}, fixed-domain asymptotics is more natural because the definition domain of the Gaussian process is fixed, and more and more observation points are added sequentially. Similarly, Gaussian processes conditioned by inequality constraints (for instance boundedness or monotonicity) can be studied under fixed-domain asymptotics \cite{lopera17finite,bachoc2018maximum}, but much less so under increasing-domain asymptotics, because the probability that a Gaussian process, defined on an unbounded domain, satisfy these types of constraints is generally zero. 

Hence, generally speaking, it seems that fixed-domain asymptotics is perhaps more often relevant than increasing-domain asymptotics, as pointed out in the reference book \cite{Ste1999}. On the other hand, the obtained theoretical results can be significantly more general, in terms of types of covariance functions and observation point locations, under increasing-domain asymptotics. Thus the joint study of both these frameworks is justified.

Finally, an ideal scenario is when both the increasing and fixed-domain asymptotic frameworks yield practical procedures that coincide, on fixed finite data sets. For instance, in \cite{Pronzato2012}, the asymptotic normal approximation given by the Fisher information matrix is used to derive efficient sampling schemes for covariance parameter estimation and prediction. Here, the Fisher information matrix is computed from finite sample data sets, so that the algorithm is motivated by both the increasing and fixed-domain frameworks. In this example, in order to be robust to the type of asymptotic setting used, we recommend to indeed compute the finite sample Fisher information matrix, rather than its theoretical limit under a given asymptotic framework.
\end{remark}

\section{Numerical experiments \label{sec:exp}}

The main goal of this section is to compare
the finite sample behavior of the WPMLE and WPCMLE of the microergodic
parameter of the exponential covariance model  
in \eqref{exp:cov}
with the asymptotic distributions given in
Theorem 3, using the MLE as a benchmark.

In a first simulation study we consider
a set of points in $[0,1]$ defined by
$s_{i+1}=s_i+0.02/L$ with $s_1=0$ and $s_n=1$. We let $L=1,2,4,8,16$, that is we consider $n=51, 101, 201, 401, 801$ points.

For each $L$, with the Cholesky decomposition, we simulate $5 000$ realizations of a zero mean Gaussian process with covariance     (\ref{exp:cov})
setting
$\sigma_0^2=1$ and $\theta_0=15$.
For each simulation we estimate with MLE, 
WPMLE  and WPCMLE
choosing the 'optimal' weights $w_1=1$ and $w_k=0$ if $k\geq2$.

Optimization was carried out using the \textsf{R} \cite{R:2005} function \texttt{optim}
with the method 
\text{"L-BFGS-B"} \cite{BB95} which allows box constraints, that is each variable can be given a lower and/or upper bound.
Specifically we set
$[a,b] \times [c,d]=[0.01,2500] \times [0.01,5]$.
From Theorems \ref{thm:consistency:inconsistency:pl} and \ref{thm:consistency:inconsistency:plc}, under this setting, both the WPMLE  and WPCMLE are consistent. Using the asymptotic distributions stated in Theorem 3, Table~\ref{tab1} compares
%Tables~\ref{tab21}, \ref{tab22} and~\ref{tab23} compare
the sample quantiles of order $0.05, 0.25, 0.5, 0.75, 0.95$,
mean, variance and kurtosis of the vector for $i \in \{1,\ldots, 5000\}$
%$[\sqrt{n}(\widehat{\sigma}_i^2(x)/ x^{1+2\kappa} - \sigma_0^2/\beta_0^{1+2\kappa} )/
%  (\sqrt{2}\sigma_0^2/\beta_0^{1+2\kappa)}]_{i_=1}^{1000}$
$$\frac{\sqrt{n} }{C_\mathrm{y}}\left( \frac{ \hat{\sigma}_{i,\mathrm{y}}^{2}\hat{\theta}_{i,\mathrm{y}}}{\sigma^{2}_{0}\theta_{0}}-1\right),$$
where $\mathrm{y}=$MLE, 
WPMLE  and WPCMLE
with the
associated theoretical values of the standard Gaussian distribution.
Here  $C^2_{MLE}=2$ and $C^2_{WPMLE }=C^2_{WPCMLE }=\widehat{\tau}^2_{n}/ (\sum_{k=1}^K w_k)^2$
where $\widehat{\tau}^2_{n}$ is computed  using the approximation in \eqref{eq:tauKn:equivalent},
i.e.,
$$
\widehat{\tau}_{n}^2
= 
\frac{2}{n}
\sum_{i=1}^{n-1} \sum_{j=i+1}^{{\rm min}(n,i+K)}
\sum_{k=1}^{n-1} \sum_{\ell=k+1}^{{\rm min}(n,k+K)}
w_{j-i} w_{\ell-k}
b_{i,j,k,\ell}
.$$

%Table 1 here

\begin{table}
		\caption{Sample quantiles, mean, variance and kurtosis of   $(\sqrt{n} / C_\mathrm{y})\left(  \hat{\sigma}_{i,\mathrm{y}}^{2}\hat{\theta}_{i,\mathrm{y}} / \sigma^{2}_{0}\theta_{0}-1\right),\quad i \in \{1,\ldots,5000\}$ for $\mathrm{y}=$ WPMLE, WPCMLE, MLE
		when $n=51, 101, 201, 401, 801$, compared with the associated theoretical
		values of the standard Gaussian distribution.}\label{tab1}
	\begin{footnotesize}
		\begin{center}
		\begin{tabular}{|c|c| c| c| c| c| c| c| c| c| }
			\hline
		Method	&$n $ & $5\%$ & $25\%$ & $50\%$ & $75\%$ & $95\%$ &mean&var&Kur\\      
			\hline
			\multicolumn{10}{c}{} \\
			 \hline
			\multirow{5}{*}{WPMLE} & 51 &-1.6186 &-0.7006 & 0.0264  &0.8950 & 2.3767 &0.1703& 1.5644 &1.6348 \\        
		&	101 &  -1.5899 &-0.6864 & 0.0278 & 0.8043 & 2.0501 & 0.1075 & 1.2520 & 0.2967\\      
		&	201 &  -1.5733& -0.6702& -0.0029&  0.7400& 1.8353 & 0.0497 & 1.0862 &0.0933 \\       
		&	401&   -1.6121& -0.6767&  0.0171 &  0.7275 & 1.7956 & 0.0419& 1.0635& -0.0321 \\      
		&	801&   -1.6275 &-0.6728 & 0.0099&  0.7229 & 1.7001& 0.0303 &1.0350& -0.0133 \\      
						\hline
			\multicolumn{10}{c}{} \\
			 \hline
						\multirow{5}{*}{WPCMLE} & 51  &-1.6186 &-0.7006 & 0.0264  &0.8950 & 2.3767 &0.1703 &1.5644& 1.6348\\        
			&	101  &-1.5899 &-0.6864 & 0.0278 & 0.8043 & 2.0501 & 0.1075 & 1.2520 &0.2967\\      				&	201   &-1.5733& -0.6702& -0.0029&  0.7400& 1.8353 & 0.0497 & 1.0862 &0.0933\\      				&	401   &-1.6121& -0.6767&  0.0171 & 0.7275 & 1.7956 & 0.0419& 1.0635 &-0.0321 \\     				&	801  & -1.6275 &-0.6728 & 0.0099&  0.7229 & 1.7001& 0.0303 &1.0350& -0.0133 \\    
					\hline
			\multicolumn{10}{c}{} \\
			 \hline
	\multirow{5}{*}{MLE} & 51 & -1.5540 &-0.7043 & -0.0585  & 0.6733 & 1.8372 &0.0197 &1.0891 &0.3874  \\      
			&	101 &-1.5556 &-0.7173 &-0.0289 & 0.6759 & 1.8059 & 0.0215 & 1.0687 &0.0327\\     
			&	201  &-1.5702& -0.7038& -0.0476 & 0.6478 & 1.7007 & -0.0034 &1.0062 &0.0316 \\					&	401  &-1.6194& -0.6909&-0.0254 & 0.6770 & 1.7012 & 0.0048 &1.0288 &-0.0607 \\     	
			&	801 & -1.6226 &-0.6886 &-0.0074&  0.6880 & 1.6634& 0.0052 &1.0195& -0.0049\\      
			\hline
			\multicolumn{10}{c}{} \\
			 \hline		
			\multicolumn{2}{|c|}{$\mathcal{N} (0,1)$}&	-1.6448& -0.6745&  0.0000&  0.6745&  1.6448 &    0 &       1  &     0\\
			\hline	
		\end{tabular}
	\end{center}
	\end{footnotesize}
\end{table}

Under the optimal weights setting and following Remark 1 we have  $C^2_{WPMLE}=C^2_{WPCMLE} \approx 2$ as in the MLE case.
\begin{center}
\begin{figure}
	\includegraphics[width=0.33\textwidth]{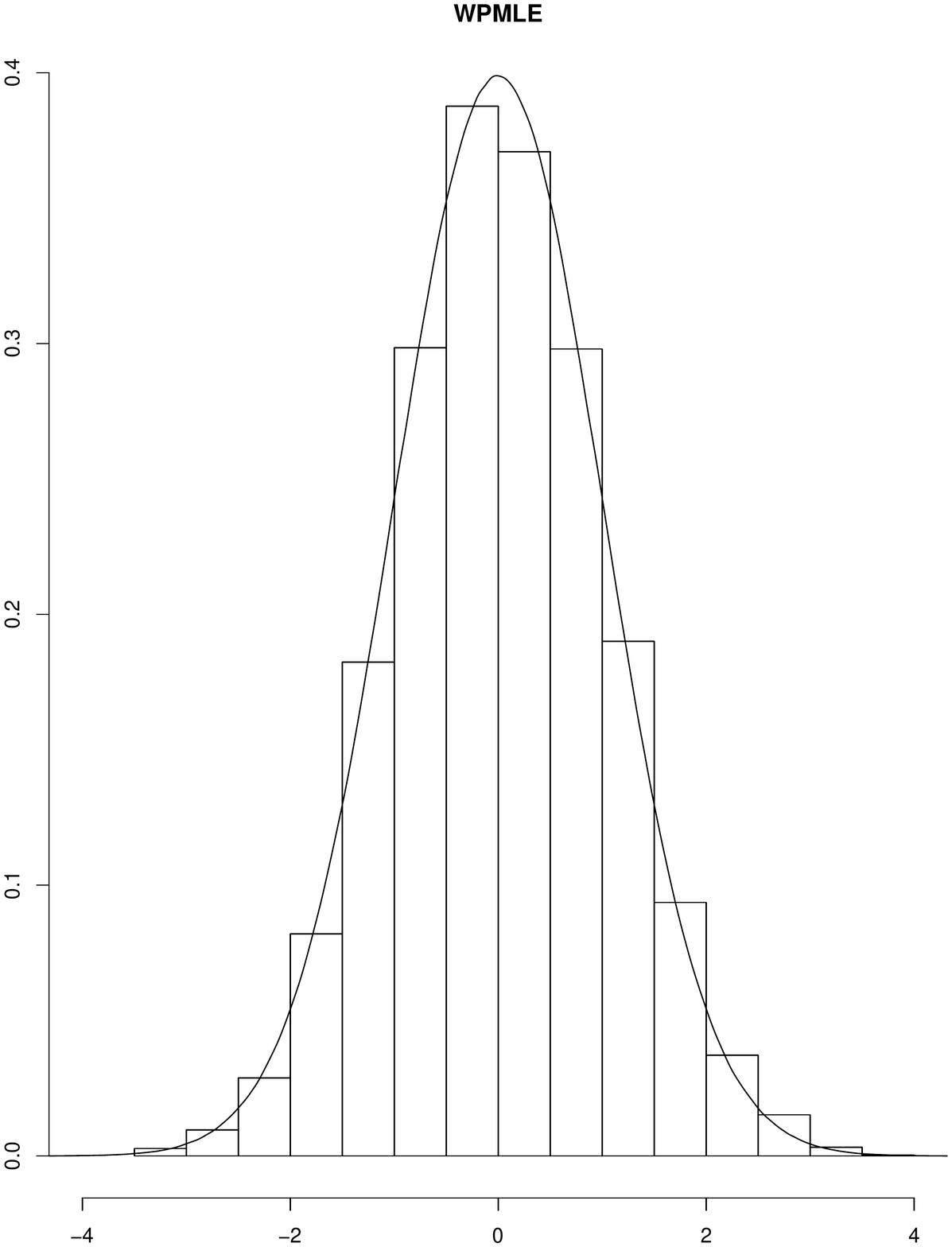}%
	\hfill%
	\includegraphics[width=0.33\textwidth]{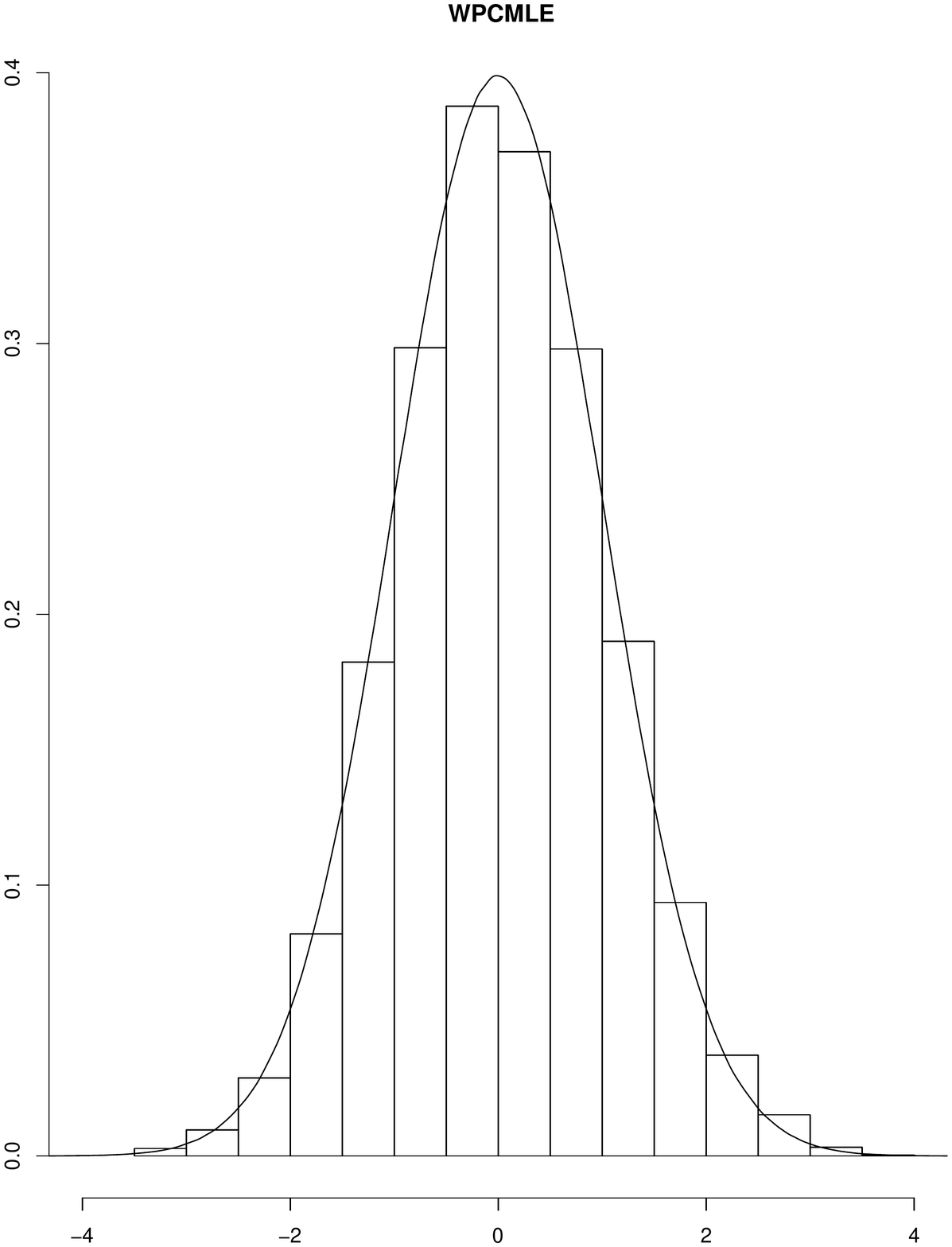}%
	\hfill%
	\includegraphics[width=0.33\textwidth]{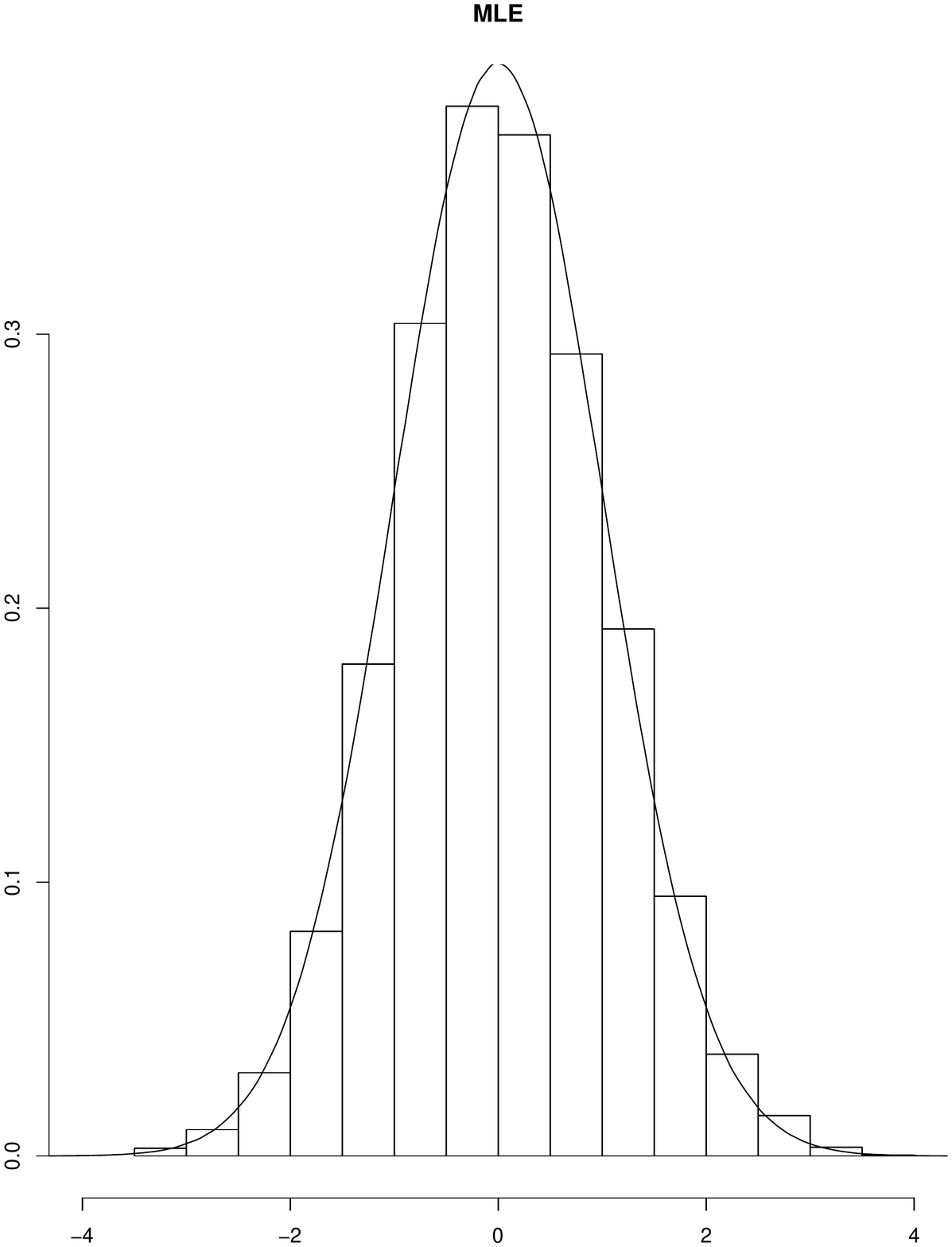}
	\caption{Comparison of the histogram of 
		$(\sqrt{n} / C_\mathrm{y}) \left(  \hat{\sigma}_{i,\mathrm{y}}^{2}\hat{\theta}_{i,\mathrm{y}} / \sigma^{2}_{0}\theta_{0} -1\right),\quad i \in \{1,\ldots,5000\}$ for $\mathrm{y}=$ WPMLE, WPCMLE, MLE (from left to right)
		with the standard Gaussian distribution when $n=801$.}\label{cdf}
\end{figure}
\end{center}

It is apparent from Table~\ref{tab1} that the asymptotic distribution given in  Theorem 3 is a satisfactory approximation of the sample
distribution, visually improving when increasing $n$.
It can be appreciated that the speed of convergence  is  slightly faster for MLE
and that the 
asymptotic behavior of the WPMLE and WPCMLE is similar.
Fig. 1 depicts a graphical comparison when $n=801$.

In a second simulation study we focus only on the asymptotic variance 
$n^{-1}(\sigma_0^2\theta_{0}C_{\mathrm{y}})^2$ and we compare it
with the sample variance of the WPMLE, WPCMLE and MLE of the microergodic parameter, under the same setting as in the first numerical experiment.
In this case we consider increasing values of both $K$ and $n$
that is we consider also non optimal weights.
The results are depicted in  Table \ref{tab3}.
For fixed $K$ 
it can be appreciated that, as expected,  the difference
between the 
sample  and theoretical variance is reduced when increasing $n$.
On the other hand, for fixed $n$, both asymptotic and simulated variances increase when increasing $K$.
Overall, as expected,  the asymptotic variance slightly underestimates the simulated variance.
Finally note that when $K\neq 1$ then the speed of convergence for the WPCMLE
is slightly faster than for the WPMLE.

%this discrepancy tends to increase with $K$, in particular when $n$
%is small. This suggests that the choice of non optimal weights affects the speed of convergence of the asymptotic approximation
%in Theorem  \ref{thm:normality}.

Let us conclude this section with a remark. For a fixed $\theta$, the global minimizer of ${p\ell}_{n}(\theta,\sigma^2)$ and of ${pc\ell}_{n}(\theta,\sigma^2)$, over $\sigma^2 \in (0,+\infty)$, can be expressed explicitly. Hence, from a computational point of view, it is possible to compute explicitly ${\rm min_{\sigma^2 \in (0,+\infty)}} {p\ell}_{n}(\theta,\sigma^2)$ and ${\rm min_{\sigma^2 \in (0,+\infty)}} {pc\ell}_{n}(\theta,\sigma^2)$. These two latter functions can be minimized numerically over $\theta$ only, which is computationally beneficial. However, this case corresponds to the case (iii) in Theorem \ref{thm:consistency:inconsistency:pl} so the WPMLE would be inconsistent in this setting.

\begin{table}
		\caption{Asymptotic variance  versus sample variance of the WPMLE, WPCMLE and MLE of the microergodic parameter when increasing $n$ and $K$.}\label{tab3}
	\begin{center}
		\begin{footnotesize}
			\begin{tabular}{|c|c|c|c|c|c|c|}
				\hline
				& $K$ &$n=$51 &$n=$101&$n=$201&$n=$401&$n=$801\\
				\hline
			\multicolumn{7}{c}{} \\
			 \hline
				Sample Variance  MLE& * &10.0650&4.8688  & 2.2809& 1.1523 & 0.5749 \\
				Asymptotic Variance &\multirow{3}{*}{}&8.8230& 4.4554&2.2388&1.1221&0.5618\\
				
			\hline
			\multicolumn{7}{c}{} \\
			 \hline
				Sample Variance WPMLE&1&13.5326&  5.5227&2.4197&  1.1905& 0.5807 \\
				Sample Variance WPCMLE&&13.5326& 5.5227&2.4197   &1.1905  & 0.5807 \\
				Asymptotic Variance && 8.6505&4.4113 &2.2277&1.1193&0.5611\\
				\hline
			\multicolumn{7}{c}{} \\
			 \hline
				Sample Variance WPMLE&10&33.5086& 17.9724&  7.9901  &3.6650 &1.7723  \\
				Sample Variance WPCMLE&&33.1477&17.7529 &  7.9465&  3.6605& 1.7720  \\
				Asymptotic Variance &&  22.2798 &12.3865  & 6.5138  &3.3382  &  1.6895 \\
				\hline
			\multicolumn{7}{c}{} \\
			 \hline
				
				Sample Variance WPMLE &20&36.1910 &24.7919 & 15.0243&7.1731  & 3.3934 \\
				Sample Variance WPCMLE&& 35.9582 &24.4611 &   14.8164 &7.1312  &  3.3894\\
				Asymptotic Variance && 32.9936     & 20.7760 & 11.5892  & 6.0239  & 3.0823    \\
					\hline
			\multicolumn{7}{c}{} \\
			 \hline
				Sample Variance WPMLE&30&  36.6338&   26.7294 & 19.7646&10.8759 &5.1132 \\
				Sample Variance WPCMLE&& 36.4643&   26.4491& 19.4406 &  10.7569 &5.0969\\
				Asymptotic Variance  &  &36.4154    &25.3444   & 16.0245 &8.6089 &  4.4547\\
				\hline
				
			\end{tabular}
		\end{footnotesize}
		\end{center}
\end{table}

\section{Conclusions\label{sec:conc}}

In this paper, we provide, to the best of our knowledge, the first fixed-domain asymptotic properties of weighted pairwise likelihood estimators. We consider the exponential covariance function in dimension one. This enables us to exploit the results of \cite{Yin1991}. By means of these results, and of specific new proof techniques, the following conclusions are obtained. The WPMLE can be inconsistent, if the range of allowed values for the variance is too large. This is in contrast with the increasing-domain asymptotic situation, and allows us to issue a warning for practical use. On the other hand, the WPCMLE is always consistent. In case of consistency, both these estimators are asymptotically Gaussian distributed, with an asymptotic variance that is larger than or equal to that of the MLE. This variance depends on the observation point locations and on the weight configurations.

Some of the conclusions obtained in this paper are specific to the exponential family of covariance functions. In particular, the MLE is here asymptotically equal to a WPCMLE based only on pairs of consecutive observation points.
On the other hand, the inconsistency of the WPMLE that we have pointed out is not specific to the exponential covariance function. Indeed, \ref{appendix:inconsistency} provides inconsistency results for the WPMLE for general covariance functions.

It would be interesting to study the consistency of the WPCMLE and its asymptotic normality for more general families of covariance functions than the exponential one, and in dimension larger than one. 
The present paper constitutes a first step in this more general direction. 
We remark that, in the maximum likelihood case, general results on fixed-domain asymptotics and for $d=1,2,3$ can be found in \cite{ShaKau2013} and  \cite{Zha2004}   for the Mat{\'e}rn model and in  \cite{Bevilacqua_et_al:2018} for the Generalized Wendland model.

As a referee pointed out, other types of CL functions have been proposed in the literature for the estimation of the covariance functions of  Gaussian processes \cite{Stein2004,Eidsviketal2014,kagugo2018}.
\cite{Lindsay:Yi:Sun:2011} highlights that, for a given estimation problem, the choice of a suitable CL function should be driven by statistical and computational considerations.
In this paper we focused only on  CL based in pairs since, from our perspective, there is a clear computational advantage when  considering CL  based only on pairs of observations.
Moreover CL methods based on pairs are more robust (in the sense of \cite{Varin:Reid:Firth:2011} and \cite{XU247})
with respect to other types of CL functions. Indeed, they require only the correct specification of bivariate random vectors.
Nevertheless, we believe that for a one dimensional Gaussian process with exponential covariance function, the results in our paper and in \cite{Yin1991}
are the starting points to study the consistency and the asymptotic distribution of other types of CL estimators under fixed-domain asymptotics.

Finally, it would be interesting to study the fixed-domain asymptotic properties of CL counterparts of the model selection criteria of \cite{chang2014asymptotic}, that are based on the full likelihood.

\section*{Acknowledgments}

The research work conducted by Moreno Bevilacqua  was supported in
part  by FONDECYT grant 1160280, Chile
and by Millennium
Science Initiative of the Ministry
of Economy, Development, and
Tourism, grant "Millenium
Nucleus Center for the
Discovery of Structures in
Complex Data".

We thank the editor in chief, the associate editor and two anonymous referees for their constructive feedback and helpful suggestions.

\appendix

\section{Proofs for Sections \ref{sec:estimators} to \ref{sec:AN}}
\label{appendix:proofs}

In the proofs, we let $0 <C_{\rm inf} < + \infty$ and $0 < C_{\rm sup} < + \infty$ be generic constants, which values can change from line to line.

\begin{proof}[{\bf Proof of Lemma \ref{lemma:for:PL:CL:to:lik}}]
	
	The principle of the proof is a change of indexation of the pairs of the form $s_i,s_j$, $i < j$. Instead of indexing such a pair by $i$ and $j$, we index it by $j-i$ and by the remainder and quotient of the Euclidean division of $i$ by $j-i$. We have
	\begin{align*}
	{p\ell}_n ( \psi ) 
	& = 
	\sum_{i=1}^{n-1}
	\sum_{j=i+1}^n
	w_{i,j}
	\ell_{s_i,s_j}(\psi)
	=
	\sum_{k=1}^{n-1}
	\sum_{i=1}^{n-1}
	w_{k}
	\mathbf{1}_{ i+k \leq n }
	\ell_{s_i,s_{i+k}}(\psi)
	\\
	& =
	\sum_{k=1}^{n-1}
	w_{k}
	\sum_{a=0}^{k-1}
	\sum_{j=0}^n
	\mathbf{1}_{ 1+a+jk+k \leq n }
	\ell_{s_{1+a+jk},s_{1+a+jk+k}}(\psi)
	=
	\sum_{k=1}^{n-1}
	w_{k}
	\sum_{a=0}^{k-1}
	\sum_{j=0}^{ \lfloor \frac{n-1-a-k}{  k} \rfloor }
	\ell_{ x^{(k,a)}_{j} , x^{(k,a)}_{j+1} }(\psi).
	\end{align*}
	The proof of \eqref{eq:pairwise:cond:lik:to:lik} is the same, after observing that
	\[
	\sum_{i=1}^{n} \sum_{\substack{j =1 \\ j \neq i}}^n
	w_{i,j} \ell_{ s_j | s_i }(\psi)
	=
	\sum_{i=1}^{n-1} \sum_{j = i+1}^n
	w_{i,j}	
	\left\{
	\ell_{ s_j | s_i }(\psi)
	+
	\ell_{ s_i | s_j }(\psi)
	\right\}.
	\]\\
\end{proof}

\begin{lemma} \label{lem:constr:GP}
	Let $m$ and $M$ be fixed with $ - \infty \leq m  < M \leq + \infty$. Then, we have, with probability $\Pr >0$ (not depending on $n$),
	\[
	\forall t \in [0,1], m \leq Z(t) \leq M.
	\]
\end{lemma}
\begin{proof}[{\bf Proof}]
	The lemma is a special case of Lemma A.3 in \cite{lopera17finite}.
\end{proof}

\begin{proof}[{\bf Proof of Theorem \ref{thm:consistency:inconsistency:pl}}]
	Let us first consider the case (i). We have, from Lemma \ref{lemma:for:PL:CL:to:lik}
	\begin{eqnarray} \label{eq:reindexing:pl}
	{p\ell}_{n}(\psi)
	& = &  
	\sum_{k=1}^{n-1}
	w_{k}
	\sum_{a=0}^{k-1}
	\sum_{j=0}^{ \lfloor \frac{n-1-a-k}{  k} \rfloor }
	\ell_{ x^{(k,a)}_{j} ,x^{(k,a)}_{j+1} }(\psi)
	\notag \\ 
	& =&
	\sum_{k=1}^{K}
	w_{k}
	\sum_{a=0}^{k-1}
	\sum_{j=0}^{ \lfloor \frac{n-1-a-k}{  k} \rfloor }
	\left(
	\ln(\sigma^2) + \frac{1}{\sigma^2} 
	Z \left( x^{(k,a)}_{j} \right)^2	 	 
	\right) + \sum_{k=1}^{K}
	w_{k}
	\sum_{a=0}^{k-1}
	\sum_{j=0}^{ \lfloor \frac{n-1-a-k}{  k} \rfloor }
	\ell_{ x^{(k,a)}_{j+1} |x^{(k,a)}_{j} }(\psi).
	\end{eqnarray}
	We observe that, for any $k \in \{ 1,\ldots,K\}$ and $a \in \{0,\ldots,k-1\}$ the points $x^{(k,a)}_{0},\ldots,x^{(k,a)}_{n_{k,a}}$, with $n_{k,a} = \lfloor (n-1-a-k)/ k \rfloor = (n/k) (1 + o(1))$ satisfy the conditions of the setting of \cite{Yin1991} (they have increasing two-by-two distinct values and are restricted to $[0,1]$). Furthermore, the set $\{ k \in \{1,\ldots,K\}, a  \in \{ 0,\ldots,k-1\} \}$ is finite with cardinality not depending on $n$. Hence it can be seen that the proof of Theorem 1 in \cite{Yin1991} (see in particular the offline equation at the top of Page 289 and Lemma 4) leads to 
	{\small
		\begin{flalign*}
		\sum_{k=1}^{K}
		w_{k}
		\sum_{a=0}^{k-1}
		\sum_{j=0}^{ \lfloor \frac{n-1-a-k}{  k} \rfloor }
		\ell_{ x^{(k,a)}_{j+1} |x^{(k,a)}_{j} }(\psi)
		-
		\sum_{k=1}^{K}
		w_{k}
		\sum_{a=0}^{k-1}
		\sum_{j=0}^{ \lfloor \frac{n-1-a-k}{  k} \rfloor }
		\ell_{ x^{(k,a)}_{j+1} |x^{(k,a)}_{j} }(\bar{\psi})
		=
		\left[
		\sum_{k=1}^{K}
		w_{k}
		\sum_{a=0}^{k-1}
		(n/k)
		\left\{
		\frac{ \theta_0 \sigma_0^2 }{ \theta \sigma^2 }
		-	 \frac{ \theta_0 \sigma_0^2 }{ \bar{\theta} \bar{\sigma}^2 } 
		+ \ln(\theta \sigma^2)	  
		- \ln (	  \bar{\theta} \bar{\sigma^2} )
		\right\}
		\right]
		+r_{1,\psi,\bar{\psi}},
		\end{flalign*}}
	where $\bar{\psi} = ( \bar{\theta} , \bar{\sigma}^2 )$
	and where ${\rm sup}_{\bar{\psi},\psi \in J} | r_{1, \psi , \bar{\psi}} | = o(n)$ almost surely. Furthermore, since $Z$ is almost surely continuous on $[0,1]$, we have almost surely
	{\normalsize
		\begin{flalign*}
		\sum_{k=1}^{K}
		w_{k}
		\sum_{a=0}^{k-1}
		\sum_{j=0}^{ \lfloor \frac{n-1-a-k}{  k} \rfloor }
		\left(
		\ln(\sigma^2) + \frac{1}{\sigma^2}
		Z \left( x^{(k,a)}_{j} \right)^2	 	 
		\right)= 
		n
		\left\{
		\sum_{k=1}^{K}
		w_{k}	
		\left(
		\ln(\sigma^2)
		+
		\frac{1}{\sigma^2}
		\frac{1}{n} \sum_{i=1}^n Z(s_i)^2	
		\right)
		\right\}
		+
		r_{2,\psi,\bar{\psi}},&
		\end{flalign*}}
	where ${\rm sup}_{\bar{\psi},\psi \in J} | r_{2, \psi , \bar{\psi}} | = o(n)$ almost surely. Hence, almost surely,
	as $n \to \infty$, we have
	{\normalsize
		\begin{flalign*}
		& {p\ell}_{n}(\psi)
		- 
		{p\ell}_{n}(\bar{\psi})
		= n
		\left(
		\sum_{k=1}^K
		w_k
		\right) 
		\left\{
		\frac{ \theta_0 \sigma^2_0 }{ \theta \sigma^2 } - 
		\frac{ \theta_0 \sigma^2_0 }{ \bar{\theta} \bar{\sigma}^2 } 
		+ \ln(\theta \sigma^2)	  
		- \ln(	  \bar{\theta} \bar{\sigma^2} )\right.&\\ 
		& \quad \quad \quad \quad \quad \quad \quad \quad \left.+ \ln(\sigma^2)+
		\frac{1}{\sigma^2}
		\left(
		\frac{1}{n} \sum_{i=1}^n Z(s_i)^2	
		\right)
		- \ln(\bar{\sigma}^2) -
		\frac{1}{\bar{\sigma}^2}
		\left(
		\frac{1}{n} \sum_{i=1}^n Z(s_i)^2	
		\right)
		\right\}
		+r_{3,\psi,\bar{\psi}}, & 
		\end{flalign*}}
	where ${\rm sup}_{\bar{\psi},\psi \in J} | r_{3, \psi , \bar{\psi}} | = o(n)$ almost surely.
	Let $g(u) = (\theta_0 \sigma_0^2) / u  + \ln(u)$,
	$  S = (1 / n) \sum_{i=1}^n Z(s_i)^2	$ and $g_S(u) = \ln(u) + S/u$. We have 
	{\normalsize
		\begin{equation}  \label{eq:for:cons:asymptotic:pl}
		{p\ell}_{n}(\psi)
		- 
		{p\ell}_{n}(\bar{\psi})
		= n
		\left(
		\sum_{k=1}^K
		w_k
		\right)
		\left\{
		g(\theta \sigma^2)
		-
		g(\bar{\theta} \bar{\sigma}^2)
		+
		g_S(\sigma^2)
		- 
		g_S(\bar{\sigma}^2)
		\right\}
		+ r_{3,\psi,\bar{\psi}} .
		\end{equation}}
	We have
	\[
	0
	<
	C_{\rm inf}
	\leq
	{\rm inf}_{ u \in [ ac , bd ] }  g(u) 
	\leq
	{\rm sup}_{ u \in [ ac , bd ] }  g(u) 
	\leq
	C_{\rm sup}
	<
	\infty.
	\]
	Consider now the subcase of (i) where $ad > \theta_0 \sigma_0^2$. Let $\epsilon >0$, and consider $\theta,\sigma^2$ so that $|\theta \sigma^2 - \theta_0 \sigma_0^2| \leq \epsilon$. We have
	\[
	\sigma^2
	\leq
	\frac{\theta_0 \sigma_0^2}{ \theta }
	+
	\frac{\epsilon}{\theta}
	\leq
	\frac{\theta_0 \sigma_0^2}{a}
	+
	\frac{\epsilon}{a}.
	\]
	Hence, for $\epsilon > 0 $ small enough there exists a fixed $d' < d$ so that, for $|\theta \sigma^2 - \theta_0 \sigma_0^2| \leq \epsilon$ we have $\sigma^2 \leq d'$.
	It follows that, with $\psi_1 = (b,d)$,
	from \eqref{eq:for:cons:asymptotic:pl}
	
	\begin{flalign*}
	\frac{1}{n \sum_{k=1}^K w_k}
	\mathrm{\lim \rm inf}_{n \to \infty}
	\left\{
	{\rm inf_{ \psi \in J ;  | \theta \sigma^2 - \theta_0 \sigma_0^2  | \leq \epsilon }}
	{p\ell}_{n} (\psi)
	-
	{p \ell}_n( \psi_1 )
	\right\} 
	& \geq 
	C_{\rm inf}
	+
	\ln(c)
	+
	\frac{1}{d'}
	S
	-
	C_{\rm sup}
	- \ln(d) 
	- \frac{1}{d}
	S 
	 \\
	&
	=
	\left( \frac{1}{d'}
	- \frac{1}{d}
	\right)
	S
	+ 
	C_{\rm inf} - C_{\rm sup} + \ln(c) - \ln(d).
	\end{flalign*}

	From Lemma \ref{lem:constr:GP}, we can show that, with probability $\Pr  >0$ (not depending on $n$)
	\[
	{\rm inf}_{t \in [0,1]} Z^2(t)
	\geq
	2
	\frac{
		- C_{\rm inf} + C_{\rm sup} - \ln(c) + \ln(d)
	}{
		\frac{1}{d'} - \frac{1}{d}
	}.
	\]

	Hence, with probability $\Pr  >0$ we have, for fixed $\epsilon >0$
	\begin{equation} \label{eq:cons:i:one}
	\frac{1}{n \sum_{k=1}^K w_k}
	\mathrm{\lim \rm inf}_{n \to \infty}
	\left\{
	{\rm inf}_{ \psi \in J ;  | \theta \sigma^2 - \theta_0 \sigma_0^2  | \leq \epsilon }
	{p\ell}_{n} (\psi)
	-
	{p\ell}_n( \psi_1 )
	\right\} 
	>0.
	\end{equation}
	This implies that $\hat{\theta}_{{p\ell}} \hat{\sigma}^{2}_{{p\ell}}$ does not converge in probability to $\theta_0 \sigma_0^2$ in the case where $ad > \theta_0 \sigma_0^2$.
	
	Let us now consider the subcase of (i) where $ bc < \theta_0 \sigma_0^2 $. Let $\epsilon >0$, and consider $\theta,\sigma^2$ so that $|\theta \sigma^2 - \theta_0 \sigma_0^2| \leq \epsilon$. We have
	\[
	\sigma^2
	\geq
	\frac{\theta_0 \sigma_0^2}{ \theta }
	-
	\frac{\epsilon}{\theta}
	\geq
	\frac{\theta_0 \sigma_0^2}{b}
	-
	\frac{\epsilon}{a}.
	\]
	Hence, for $\epsilon > 0 $ small enough there exists a fixed $c' > c$ so that, for $|\theta \sigma^2 - \theta_0 \sigma_0^2| \leq \epsilon$ we have $\sigma^2 \geq c'$. One can check that the function $g_S$ in \eqref{eq:for:cons:asymptotic:pl} has a second derivative which is negative for $t \geq 2S$. Hence, if $\epsilon \leq c'/4$ and $S \leq c'/4$ we have, since $\sigma^2 \geq c'$
	\begin{align*}
	g_S(\sigma^2) - g_S(\sigma^2 - \epsilon)
	\geq 
	\epsilon \, {\rm inf}_{ t \in [c' - \epsilon , d] } g_S'(t)
 \geq 
	\epsilon g_S'( d ) 
\geq 
	\epsilon \left(  \frac{1}{d} - \frac{S}{d^2} \right) 
 \geq
	\frac{\epsilon}{2d}
	\end{align*}
	if we further assume that $S \leq d/2$. 
	Also, we have (still considering $|\theta \sigma^2 - \theta_0 \sigma_0^2| \leq \epsilon$)
	\begin{align*}
	\left|
	g( \theta \sigma^2 )
	- 
	g( \theta (\sigma^2 - \epsilon) )
	\right|\leq
	b \, \epsilon \, {\rm sup}_{ t \in [ \theta_0 \sigma_0^2 + \epsilon , \theta_0 \sigma_0^2 - \epsilon - b \epsilon  ] }
	| g'(t) | 
	\ =
	b \, \epsilon \, {\rm sup}_{ t \in [ \theta_0 \sigma_0^2 + \epsilon , \theta_0 \sigma_0^2 - \epsilon - b \epsilon  ] }
	\left|
	\frac{1}{t} \left( 1 - \frac{\sigma_0^2 \theta _0}{t} \right) 
	\right| 
	\leq
	A_{\rm sup} \epsilon^2
	\end{align*}
	when $\epsilon \leq \nu$, where $A_{\rm sup} < \infty$ and $\nu >0$ can be chosen so as to depend only on $\theta_0,\sigma_0^2,b$. Hence, we have, for $\epsilon \leq c'/4$, $\epsilon \leq \nu$, $S \leq c'/4$ and $S \leq d/2$,
	\[
	g( \theta \sigma^2 ) - g( \theta \{ \sigma^2 - \epsilon \} )
	+ 
	g_S(\sigma^2) - g_S(\sigma^2 - \epsilon)
	\geq 
	\frac{\epsilon}{2d}
	-
	A_{\rm sup} \epsilon^2.
	\]
	From Lemma \ref{lem:constr:GP}, we can show that with probability $\Pr >0$, not depending on $n$, we have ${\rm sup}_{t \in [0,1]} Z^2(t) \leq  \rm min(c' / 4,d/2)$. 
	Hence, we can choose $\epsilon >0$, not depending on $n$ so that, from \eqref{eq:for:cons:asymptotic:pl}, with probability $\Pr > 0$
	\begin{equation} \label{eq:cons:i:two}
	\frac{1}{n \sum_{k=1}^K w_k}
	\mathrm{\lim \rm inf}_{n \to \infty}
	\left[
	{\rm inf_{ \psi \in J ;  | \theta \sigma^2 - \theta_0 \sigma_0^2  | \leq \epsilon }}
	{p\ell}_{n} (\psi)
	-
	{p\ell}_n\{ (\theta,\sigma^2 - \epsilon) \}
	\right]
	>0.
	\end{equation}
	This implies that $\hat{\theta}_{{p\ell}} \hat{\sigma}^{2}_{{p\ell}}$ does not converge in probability to $\theta_0 \sigma_0^2$ in the case where $bc < \theta_0 \sigma_0^2$. Hence, the proof of the case (i) is concluded. The displays \eqref{eq:cons:i:one} and \eqref{eq:cons:i:two} also each imply the case (iii). 
	
	Let us now address the case (ii). Let $\tilde{\sigma}^2$ be defined as
	\[
	\tilde{\sigma}^2
	=
	\begin{cases}
	S ~ ~ & \mbox{if} ~ ~ S \in [c,d] \\
	c ~ ~ & \mbox{if} ~ ~ S \leq c \\
	d ~ ~ & \mbox{if} ~ ~ S \geq d. \\
	\end{cases}
	\]
	Let $\tilde{\theta} = (\theta_0 \sigma_0^2) / \tilde{\sigma}^2$ and observe that by assumption
	\[
	\tilde{\theta} 
	\in
	\left[  
	\frac{\theta_0 \sigma_0^2}{d},
	\frac{\theta_0 \sigma_0^2}{c}
	\right]
	\subset
	\left[  
	\frac{ad}{d},
	\frac{bc}{c}
	\right]
	= [a,b].
	\]
	Then, from \eqref{eq:reindexing:pl}, from the last display of the proof of Theorem 1 in \cite{Yin1991} and from similar arguments as before, we have for $\epsilon >0$, almost surely, with $\tilde{\psi} = (\tilde{\theta},\tilde{\sigma}^2) \in J$ with $\tilde{\theta}\tilde{\sigma}^2 = \theta_0 \sigma_0^2$, 
	\begin{flalign*}
	&\frac{1}{n \sum_{k=1}^K w_k}
	\mathrm{\lim \rm inf}_{n \to \infty}
	{{\rm inf}_{ \substack{ | \theta \sigma^2 - \tilde{\theta} \tilde{\sigma}^2| \geq \epsilon \\ (\theta,\sigma^2) \in J} }}
	\left\{
	{p\ell}_{n}( \psi )
	-
	{p\ell}_{n}( \tilde{\psi} )
	\right\}
\geq 
	{\rm inf}_{ \substack{ | \theta \sigma^2 - \tilde{\theta} \tilde{\sigma}^2| \geq \epsilon \\ (\theta,\sigma^2) \in J} }
	\left\{
	g(\theta\sigma^2)
	-g( \theta_0 \sigma_0^2 )
	+g_S( \sigma^2 )
	-g_S(\tilde{\sigma}^2)
	\right\}&
	\\
	&\quad \quad\quad\quad\quad\quad\quad\quad\quad\quad\quad\quad\quad\quad\quad\quad\quad\quad\quad\quad\quad \geq
	{\rm inf}_{ \substack{ | \theta \sigma^2 - \tilde{\theta} \tilde{\sigma}^2| \geq \epsilon \\ (\theta,\sigma^2) \in J} }
	\left\{
	g(\theta\sigma^2)
	-g( \theta_0 \sigma_0^2 )
	\right\} 
	 > 0.
	\end{flalign*}
	This concludes the proof for the case (ii).
	
	Consider now the case (iv) where $J = (0 , \infty) \times  [c,d]$. Let $\tilde{\sigma}$ and $\tilde{\theta}$ be defined as in the proof for the case (ii).
	From the offline equation after (2.23) in the proof of Theorem 1 in \cite{Yin1991}, we have, with the notation and arguments in \eqref{eq:reindexing:pl},
	
	\begin{equation} \label{eq:cons:infinite:supp:theta:one}
	\sum_{a=0}^{k-1}
	\sum_{j=0}^{ \lfloor \frac{n-1-a-k}{  k} \rfloor }
	\ell_{ x^{(k,a)}_{j+1} |x^{(k,a)}_{j} }(\psi)
	-
	\sum_{a=0}^{k-1}
	\sum_{j=0}^{ \lfloor \frac{n-1-a-k}{  k} \rfloor }
	\ell_{ x^{(k,a)}_{j+1} |x^{(k,a)}_{j} }(\tilde{\psi})
	\geq
	\frac{n}{\theta}
	\left( 
	\frac{\theta_0 \sigma_0^2}{2d}
	\right)	 
	+
	\frac{1}{\theta}
	r_{4,\psi,\tilde{\psi}},
	\end{equation}
	where ${\rm sup}_{ \theta \leq \rho, c \leq \sigma^2 \leq d } |r_{4,\psi,\tilde{\psi}}| = o(n) $ almost surely, where $\rho>0$ is a constant depending only on $\theta_0,\sigma_0^2,c,d$.
	Also, from (2.26) in the proof of Theorem 1 in \cite{Yin1991}, we have
	\begin{equation} \label{eq:cons:infinite:supp:theta:two}
	\sum_{a=0}^{k-1}
	\sum_{j=0}^{ \lfloor \frac{n-1-a-k}{  k} \rfloor }
	\ell_{ x^{(k,a)}_{j+1} |x^{(k,a)}_{j} }(\psi)
	-
	\sum_{a=0}^{k-1}
	\sum_{j=0}^{ \lfloor \frac{n-1-a-k}{  k} \rfloor }
	\ell_{ x^{(k,a)}_{j+1} |x^{(k,a)}_{j} }(\tilde{\psi})
	\geq
	\nu n
	+
	r_{5,\psi,\tilde{\psi}},
	\end{equation}
	where ${\rm sup}_{ \theta \geq B, c \leq \sigma^2 \leq d } |r_{5,\psi,\tilde{\psi}}| = o(n) $ almost surely, where $\nu >0$ and $B < \infty$ can be chosen as functions of $\theta_0, \sigma_0^2,c,d$ only. Finally, as shown for obtaining \eqref{eq:for:cons:asymptotic:pl},
	\begin{eqnarray} \label{eq:cons:infinite:supp:theta:three}
	\sum_{a=0}^{k-1}
	\sum_{j=0}^{ \lfloor \frac{n-1-a-k}{  k} \rfloor }
	\ell_{ x^{(k,a)}_{j+1} |x^{(k,a)}_{j} }(\psi)
	-
	\sum_{a=0}^{k-1}
	\sum_{j=0}^{ \lfloor \frac{n-1-a-k}{  k} \rfloor }
	\ell_{ x^{(k,a)}_{j+1} |x^{(k,a)}_{j} }(\tilde{\psi})	=
	n
	\left\{
	g(\theta \sigma^2) - g(\theta_0 \sigma_0^2)
	\right\}
	+
	r_{6,\psi,\tilde{\psi}},
	\end{eqnarray}
	where ${\rm sup}_{ \rho \leq \theta \leq B, c \leq \sigma^2 \leq d } |r_{6,\psi,\tilde{\psi}}| = o(n) $ almost surely. We have for $\epsilon >0$ ${\rm inf}_{|\theta \sigma^2 - \theta_0 \sigma_0^2|\geq \epsilon } g(\theta \sigma^2 ) - g( \theta_0 \sigma_0^2 ) >0$. Hence, there exists a constant $A_{\rm inf}>0$, not depending on $n$, so that, from \eqref{eq:cons:infinite:supp:theta:one}, \eqref{eq:cons:infinite:supp:theta:two}, \eqref{eq:cons:infinite:supp:theta:three} and \eqref{eq:reindexing:pl}
	{\normalsize
		\begin{align*}
		{\rm inf_{\substack{ \theta \in (0,\infty), \sigma^2 \in [c,d] 
				\\
				|\theta \sigma^2 - \tilde{\theta} \tilde{\sigma}^2| \geq \epsilon } }}
		{p\ell}_n(\psi)
		-
		{p\ell}_n(\tilde{\psi})
		\geq  &
		n
		\left(
		\sum_{k=1}^K
		w_k
		\right)
		\left\{
		A_{\rm inf}
		+g_S(\sigma^2)
		-g_S(\tilde{\sigma}^2)
		\right\}
		+
		r_{7,\psi,\tilde{\psi}},
		\geq  
		n
		\left(
		\sum_{k=1}^K
		w_k
		\right)
		A_{\rm inf}
		+
		r_{7,\psi,\tilde{\psi}},
		\end{align*}}
	where $\rm sup_{0 < \theta <  \infty, c \leq \sigma^2 \leq d } | r_{7,\psi,\tilde{\psi}}| = o(n) $ almost surely. This concludes the proof for the case (iv).
\end{proof}

\begin{proof}[{\bf Proof of Theorem \ref{thm:consistency:inconsistency:plc}}]
	
	Let $n_{k,a} = \lfloor (n-1-a-k)/  k \rfloor$, $R^{(k,a)}_{\psi} = [ \sigma^2 e^{- \theta | x^{(k,a)}_i - x^{(k,a)}_j  | }]_{0 \leq i,j \leq n_{k,a} +1}$ and $Z^{(k,a)} = (x^{(k,a)}_0,\ldots,x^{(k,a)}_{n_{k,a}+1})^\top$.
	We have for $\psi \in J$, from Lemma \ref{lemma:for:PL:CL:to:lik}, by writing Gaussian likelihoods as products of conditional likelihoods, forward and backward, and by using the Markovian property of $Z$ (Lemma 1 in \cite{Yin1991}),
	{\normalsize
		\begin{align*}
		{pc\ell}_n (\psi)
		= &
		\sum_{k=1}^{K}
		w_{k}
		\sum_{a=0}^{k-1}
		\sum_{j=0}^{ \lfloor \frac{n-1-a-k}{  k} \rfloor }
		\left\{
		\ell_{ x^{(k,a)}_{j+1} |x^{(k,a)}_{j} }(\psi)
		+ 
		\ell_{ x^{(k,a)}_{j} |x^{(k,a)}_{j+1} }(\psi)	 
		\right\}
		\\  
		= &	\sum_{k=1}^{K}
		w_{k}
		\sum_{a=0}^{k-1}
		\left(
		\ln(| R^{(k,a)}_{\psi} |) 
		+ (Z^{(k,a)})^\top R_{\psi}^{-1} Z^{(k,a)} 
		-	 \ln(\sigma^2) - \frac{1}{\sigma^2} Z(x^{(k,a)}_0)^2 
		\right.
		\\
		& 
		\left.
		+
		\ln(| R^{(k,a)}_{\psi} |) 
		+ (Z^{(k,a)})^\top R_{\psi}^{-1} Z^{(k,a)} 
		-	 \ln(\sigma^2) - \frac{1}{\sigma^2} Z(x^{(k,a)}_{n_{k,a} + 1})^2  \right)
		\\
		= &
		\sum_{k=1}^{K}
		w_{k}
		\sum_{a=0}^{k-1}
		\left[
		\left\{
		2 \sum_{j=0}^{ \lfloor \frac{n-1-a-k}{  k} \rfloor }
		\ell_{ x^{(k,a)}_{j+1} |x^{(k,a)}_{j} }(\psi) 
		\right\}
		+	 \ln(\sigma^2) + \frac{1}{\sigma^2} Z(x^{(k,a)}_{0})^2
		-	 \ln(\sigma^2) - \frac{1}{\sigma^2} Z(x^{(k,a)}_{n_{k,a} + 1})^2
		\right]
		\\
		= &
		2 \sum_{k=1}^{K}
		w_{k}
		\sum_{a=0}^{k-1}
		\sum_{j=0}^{ \lfloor \frac{n-1-a-k}{  k} \rfloor }
		\ell_{ x^{(k,a)}_{j+1} |x^{(k,a)}_{j} }(\psi) 
		+	 \sum_{k=1}^{K}
		w_{k}
		\sum_{a=0}^{k-1}
		\frac{1}{\sigma^2}
		\left(
		Z( s_{1+a} )^2
		-
		Z(s_{1+a+k (n_{k,a}+1)})^2
		\right).
		\end{align*}}
	Hence, we conclude the proof in the case $J = [a,b] \times [ c , d ] $ or $J = (0 , \infty) \times  [c,d]$ from \eqref{eq:cons:infinite:supp:theta:one}, \eqref{eq:cons:infinite:supp:theta:two} and \eqref{eq:cons:infinite:supp:theta:three}. In the case $J = [a,b] \times (0 , \infty)$, we can express $\hat{\sigma}^2(\theta) = \mathrm{argmin}_{\sigma^2 \in (0,\infty)} {pc\ell}_n(\theta,\sigma^2)$ explicitly, and conclude using similar techniques as in \cite{Yin1991} and similar arguments as in \eqref{eq:reindexing:pl}. We skip the details. 
\end{proof}

\begin{proof}[{\bf Proof of Lemma \ref{lemma:approx:tauKn}}]
	We have that for $1 \leq i<j\leq n$ and $1 \leq k<\ell \leq n$, $$\mathrm{\rm cov}(W_{i,j}^2-1,W_{k,\ell}^2-1 ) = 2 \mathrm{\rm  cov}(W_{i,j},W_{k,\ell})^2.$$ One can show, by computing explicitly the covariances $\mathrm{\rm  cov}(W_{i,j},W_{k,\ell})$, by distinguishing the different cases in \eqref{eq:bijkl}, and by considering Taylor expansions, that we have for $1 \leq i < j \leq {\rm min}(n,i + K)$ and $1 \leq k < \ell \leq {\rm min}(n,k + K)$,
	\[
	\mathrm{\rm cov}(W_{i,j},W_{k,\ell})^2
	=
	b_{i,j,k,\ell} + r_{i,j,k,\ell} ( s_{{\rm max}(i,j,k,\ell)} - s_{{\rm min}(i,j,k,\ell)} ),
	\]
	where 
	\[
	\rm sup_{ \substack{
			n \in \mathbb{N} 
			\\
			i,j,k,\ell \in\{ 1,\ldots,n\}
			\\
			i +1 \leq j \leq i + K
			\\
			k +1 \leq l \leq k + K
			\\
			|k-i| \leq K
	}}
	|r_{i,j,k,\ell}| = O(1)
	~ ~ ~
	\mbox{and}
	~ ~ ~
	r_{i,j,k,\ell} = 0
	~ ~ 
	\mbox{if}
	~ ~ 
	|k-i| > K.
	\]
	In the multiple sum
	\[
	\sum_{i=1}^{n-1} \sum_{j=i+1}^{{\rm min}(n,i+K)}
	\sum_{k={\rm max}(1,i-K)}^{{\rm min}(n,i+K)} \sum_{\ell=k+1}^{{\rm min}(n,k+K)}
	( s_{{\rm max}(i,j,k,\ell)} - s_{{\rm min}(i,j,k,\ell)} )
	\]
	one can see that for any $a \in \{2,\ldots,n\}$, $ \Delta_a = s_a - s_{a-1}$ appears less than $4K^4$ times. Hence, this sum is a $O \left( \sum_{a=2}^n \Delta_a  \right)$. This concludes the proof.
\end{proof}

\begin{proof}[{\bf Proof of Lemma \ref{lem:larger:variance}}]
	We have, for $n \geq K+1$,
	\[
	\tau_{n}^2
	=
	\left\{
	\frac{2}{n} \sum_{k=1}^K  \sum_{\ell=1}^K
	w_k w_\ell
	\left( \sum_{i=1}^{n-k} \sum_{j=1}^{n-\ell} b_{i,i+k,j,j+\ell}  \right)
	\right\}
	+o(1).
	\]
	Let $k,\ell \in \{1,\ldots,K\}$ be fixed. Without loss of generality, assume that $\ell \leq k$. We have, for $i\in \{ 1,\ldots,n-k\}$
	\begin{align*}
	\sum_{j=1}^{n-\ell} b_{i,i+k,j,j+\ell}
	\geq &
	\sum_{j=i}^{i+k-\ell} b_{i,i+k,j,j+\ell}
	= 
	\sum_{j=i}^{i+k-\ell} \frac{s_{j+\ell} - s_j}{s_{i+k} - s_i} 
	\geq 
	\frac{s_{i+\ell} - s_i}{s_{i+k} - s_i}
	+
	\sum_{j=i+1}^{i+k-\ell} \frac{s_{j+\ell} - s_{j+\ell-1}}{s_{i+k} - s_i} 
	=   
	\frac{s_{i+k} - s_i}{s_{i+k} - s_i}
	=   
	1.
	\end{align*}
	Hence $\sum_{i=1}^{n-k} \sum_{j=1}^{n-\ell} b_{i,i+k,j,j+\ell} \geq n-k$.
	This concludes the proof.
\end{proof}

\begin{proof}[{\bf Proof of Lemma \ref{lem:smaller:variance}}]
	The lemma follows from Lemma \ref{lemma:approx:tauKn}, because, with the notation there, each $b_{i,j,k,\ell}$ is in $[0,1]$ and, for any fixed $i$, there are less than $2 K^3$ values of $(j,k,\ell)$ in \eqref{eq:tauKn:equivalent} for which $b_{i,j,k,\ell}$ is non-zero. 
\end{proof}

\begin{proof}[{\bf Proof of Theorem \ref{thm:normality}}]
	Because of the consistency results in Theorems \ref{thm:consistency:inconsistency:pl} and \ref{thm:consistency:inconsistency:plc}, it is sufficient to consider the case where $J = [a,b] \times [c,d]$. Furthermore, we only write the proof of \eqref{dotto}, since the proof of \eqref{dotto:pcl} is similar.

	Let 
	\begin{equation}\nonumber
	\omega(\psi)=\frac{\partial}{\partial \theta} {p\ell}_{n}(\psi).
	\end{equation}
	From (\ref{eq:pairwise:lik:to:lik}), we can write
	
	\begin{eqnarray}\nonumber\label{der1}
	\omega(\psi)&=& \sum_{k=1}^{K}
	w_k
	\sum_{a=0}^{k-1}\sum_{j=0}^{\lfloor \frac{n-a-1-k}{k}\rfloor}\frac{\partial}{\partial \theta}
	\ell_{s_{jk+a+1},s_{(j+1)k+a+1}}(\psi).  
	\end{eqnarray}
	From (3.11) in \cite{Yin1991} , we can write, with $\rm sup_{\psi \in [a,b] \times [c,d]} |r_{i,\psi}| = O_p(1)$ for $i \in \{1,2,3\}$,
	{\normalsize
		\begin{eqnarray}\label{der2}
		\omega(\psi)
		&=&
		\left\{
		- \sum_{k=1}^{K} 
		w_k
		\sum_{a=0}^{k-1}\frac{\sigma_{0}^{2}\theta_{0}}{\sigma^{2}\theta^2}
		\sum_{j=0}^{\lfloor \frac{n-a-1-k}{k}\rfloor}
		\left(
		W^{2}_{jk+a+1,(j+1)k+a+1} + \left(\lfloor \frac{n-a-1-k}{k}\rfloor  \right) 
		\frac{1}{\theta}
		\right)
		\right\}
		 + r_{1,\psi}  \\
		&=&
		\left\{
		\sum_{k=1}^{K}
		w_k
		\sum_{a=0}^{k-1}
		\sum_{j=0}^{\lfloor \frac{n-a-1-k}{k}\rfloor} 
		\left(
		\frac{-\sigma_{0}^{2}\theta_{0}}{\sigma^2 \theta^2}
		W^{2}_{jk+a+1,(j+1)k+a+1} + \frac{1}{\theta}
		\right)
		\right\}
		+
		r_{2,\psi}
		\nonumber =
		\left\{
		\sum_{i=1}^{n}
		\sum_{k=1}^{\rm min\{K,n-i\}}
		w_k
		\left(
		- \frac{\sigma_{0}^{2}\theta_{0}}{\sigma^2 \theta^2 }
		W^{2}_{i,i+k} + \frac{1}{\theta}\right)   
		\right\}
		+  r_{3,\psi}.
		\end{eqnarray}}
	Now $\hat{\sigma}^{2}\in [a,b]$ and by Theorem \ref{thm:consistency:inconsistency:pl}, $
	\hat{\sigma}_{{p\ell}}^{2}\hat{\theta}_{{p\ell}} \stackrel{a.s}{\longrightarrow}
	\theta_{0}\sigma_{0}^{2}$  
	and in view of \eqref{der2}, we obtain
	
	\begin{eqnarray*}\nonumber\label{der3}
	O_p(1)
	&=& 
	\sum_{i=1}^{n}
	\sum_{k=1}^{{\rm min}\{K,n-i\}}
	w_k
	\left( -\sigma_{0}^{2}\theta_{0}W^{2}_{i,i+k} + \hat{\sigma}_{{p\ell}}^{2}\hat{\theta}_{{p\ell}}\right)   \\
	\end{eqnarray*}
	
	and
	
	\begin{flalign*}
	O_p(1)
	&=
	\sum_{i=1}^{n}\sum_{k=1}^{{\rm min}\{K,n-i\}} 
	w_k
	\left\{
	-\sigma_{0}^{2}\theta_{0}\left(W^{2}_{i,i+k}-1\right)
	\right\}
	+
	\sum_{i=1}^{n}\sum_{k=1}^{{\rm min}\{K,n-i\}}
	w_k
	(\hat{\sigma}_{{p\ell}}^{2}\hat{\theta}_{{p\ell}}-\sigma^{2}_{0}\theta_{0}).
	\end{flalign*}
	
	Then
	\begin{flalign*}
		\sqrt{n}
	\left( \sum_{k=1}^K w_k \right)
	\left\{ 1+o_p(1) \right\} \left(\hat{\sigma}_{{p\ell}}^{2}\hat{\theta}_{{p\ell}}-\sigma_{0}^{2}\theta_{0}\right) 
	= 
	\left\{
	\frac{\sigma_{0}^{2}\theta_{0}}{\sqrt{n}}
	\sum_{i=1}^{n}\sum_{k=1}^{\rm min\{K,n-i\}}  
	w_k \left(W^{2}_{i,i+k}-1\right)
	\right\}
	+o_p(1).
	&
	\end{flalign*}
	For $i\in\{1,\ldots,n\}$ and $k\in \{1,\ldots,{\rm min}(K,n-i)\}$, let
	$(W^{2}_{i,i+k}-1)= T_{(i-1)K+k}$. Then we can write
	
	\begin{eqnarray} \label{eq:for:TCL}
	\sqrt{n}
	\left( \sum_{k=1}^K w_k \right)
	\left(\hat{\sigma}_{{p\ell}}^{2}\hat{\theta}_{{p\ell}}-\sigma_{0}^{2}\theta_{0}\right) 
	&=& 
	\frac{\sigma_{0}^{2}\theta_{0}}{\sqrt{n}}
	\left(
	\sum_{i=1}^{n}\sum_{k=1}^{{\rm min}\{K,n-i\}} 
	w_k
	T_{(i-1)K+k} 
	\right)  +o_p(1).
	\end{eqnarray}
	Since $T_{i}$ is independent of $T_{j}$ for $|i-j|\geq K(K+1)$ we can apply Theorem 2.1 in \cite{Neumann2013} for weakly dependent variables and we can establish a central limit theorem for
	$ \{ (\sqrt{n} \sum_{k=1}^K w_k ) / (\tau_{n}\sigma^{2}_{0}\theta_{0} )\} \left(\hat{\sigma}_{{p\ell}}^{2}\hat{\theta}_{{p\ell}}-\sigma_{0}^{2}\theta_{0}\right)$.\\
	
	Let us assume
	\begin{flalign*}
	\frac{\sqrt{n} (\sum_{k=1}^K w_k)}{\tau_{n}\sigma^{2}_{0}\theta_{0}} \left(\hat{\sigma}_{{p\ell}}^{2}\hat{\theta}_{{p\ell}}-\sigma_{0}^{2}\theta_{0}\right)
	&\stackrel{\mathcal{D}}{\nrightarrow}
	\mathcal{N}
	(0, 1).
	\end{flalign*}
	
	By Lemmas \ref{lem:larger:variance} and \ref{lem:smaller:variance}, we can extract a subsequence $n' \to \infty $ so that $\tau_{K,n'}^{2} \to \tau_{K}^{2} \in (0 , \infty)$ as $n' \to \infty$ and so that 
	\begin{flalign} \label{eq:TCL:to:be:contradicted}
	\frac{\sqrt{n} (\sum_{k=1}^K w_k) }{\tau_{K,n'}\sigma^{2}_{0}\theta_{0}} \left(\hat{\sigma}_{{p\ell}}^{2}\hat{\theta}_{{p\ell}}-\sigma_{0}^{2}\theta_{0}\right)
	&\stackrel{\mathcal{D}}{\nrightarrow}
	\mathcal{N}
	(0, 1).
	\end{flalign}
	
	The triangular array $\left(T_i 
	\right)$ satisfies the conditions of  Theorem 2.1 in \cite{Neumann2013}, thus we obtain
	
	\begin{eqnarray*}
		\frac{1}{\sqrt{n'}}\sum_{i=1}^{n'}
		\sum_{k=1}^{\rm min\{K,n'-i\}}
		w_k  T_{(i-1)K+k}
		\stackrel{\mathcal{D}}{\longrightarrow}
		\mathcal{N}
		(0, \tau_{K}^2).  
	\end{eqnarray*}
	
	Hence from \eqref{eq:for:TCL} and Slutsky's lemma we are in contradiction with \eqref{eq:TCL:to:be:contradicted}.
	This concludes the proof.
\end{proof}

\begin{proof}[{\bf Proof of Lemma \ref{lem:strictly:larger:variance}}]
	From the proof of Lemma \ref{lem:larger:variance}, it suffices to show that for any $2 \leq k \leq K$,
	\[
	\mathrm{\lim \rm inf}_{n \to \infty}
	\frac{1}{n}
	\sum_{i=1}^{n-k} \sum_{j=1}^{n-k} b_{i,i+k,j,j+k}  >1.
	\]
	We have from \eqref{eq:bijkl}, since $s_i = i/n$ for all $n \in \mathbb{N}$ and $i\in \{1,\ldots,n-k\}$ 
	\[
	\sum_{j=1}^{n-k} b_{i,i+k,j,j+k}
	\geq
	1
	+
	b_{i,i+k,i+1,i+k+1}
	=
	1+
	\frac{ (k-1)^2 }{ k^2 }.
	\]
	Hence 
	\[
	\frac{1}{n}
	\sum_{i=1}^{n-k} \sum_{j=1}^{n-k} b_{i,i+k,j,j+k}
	\geq 
	1+
	\frac{ (k-1)^2 }{ k^2 } + o(1).
	\]
\end{proof}

\section{An inconsistency result for variance estimation by weighted pairwise log-likelihood, for general correlation functions}
\label{appendix:inconsistency}

In this section, we extend the inconsistency results (i) and (iii)
of Theorem \ref{thm:consistency:inconsistency:pl} for the WPMLE,
to more general covariance functions than the exponential ones as in \eqref{exp:cov}. Furthermore, the input space of the Gaussian process under consideration is $[0,1]^d$, where the dimension $d$ is allowed to be larger than one.
We consider the estimation of a variance parameter, where the correlation function is known. 

We consider a stationary correlation function $C : \mathbb{R}^d \to \mathbb{R}$ on $\mathbb{R}^d$, with $d \in \mathbb{N}$. We make the following assumptions on $C$.

\begin{assumption} \label{condition:C}
\normalfont
The function $C$ is continuous. 
 Let $d_{C} (\bx,\bx') = \sqrt{2 C(\bzero) - 2C(\bx  -  \bx')}$.
	Let $N([0,1]^d,d_C,\rho)$ be the minimum number of balls with radius $\rho$ (w.r.t. the distance $d_C$), required to cover $[0,1]^d$.
We have
	\begin{equation}
		\int_{0}^{\infty} \sqrt{\ln\{N([0,1]^d, d_{C},\rho)\}} \ d\rho< \infty.
		\label{eq:intBall0}
	\end{equation}
	Furthermore, the Fourier transform $\widehat{C}$ of $C$ satisfies
	\begin{equation} 
		\exists \ P < \infty \quad \mbox{so that as} \quad \|\bw\| \to \infty, \quad \widehat{C}(\bw) \|\bw\|^P \to \infty.
		\label{eq:Fourier0}
	\end{equation}
	Finally, $C(\bzero) = 1$ and $C(\bx)<1$ for $\bx \neq \bzero$.
\end{assumption}

Many stationary correlation functions satisfy the above assumption. The condition \eqref{eq:intBall0} holds when $C$ is H\"older continuous and the condition \eqref{eq:Fourier0} holds when $C$ belongs to the Mat\'ern class of covariance functions \cite{Ste1999}. We refer to the discussion after Condition A.1 in \cite{lopera17finite} for more insight on the above assumption.

We now consider a zero-mean Gaussian process $Y$ on $[0,1]^d$ with covariance function $\sigma_0^2 C$ with an unknown fixed $\sigma_0^2 \in (0,\infty) $. We assume that $C$ is known. Consider a triangular array of weights $\{ w^{(n)}_{i,j} \}_{n \in \mathbb{N}, 1 \leq i  < j \leq n}$ in $[0,\infty)$ and let $w^{(n)}_{i,j} = w_{i,j}$. Assume that for all $n \in \mathbb{N}$, not all $w_{i,j}$ are zero.
Consider also a sequence $(\bx_i)_{i \in \mathbb{N}}$ of observation points on $[0,1]^d$. Then, the pairwise log likelihood is 

\begin{equation*}
	\ell_{C, \bx_i,\bx_j }(\sigma^2)
	= 2 \ln (\sigma^{2})+ \ln (1-C(\bx_i-\bx_j)^2)+
	\frac{1}{\sigma^2} Y(\bx_i)^2
	+
	\frac{ \{Y(\bx_j)- C(\bx_i-\bx_j) Y(\bx_i)\}^2 }{\sigma^{2}(1-C(\bx_i-\bx_j)^2)}
\end{equation*}
and the weighted pairwise log likelihood criterion is

\begin{equation*}
		p\ell_{C,n}(\sigma^2) =  
		\sum_{i=1}^{n-1} \sum_{j=i+1}^n
		w_{i,j} \ell_{ C,\bx_i,\bx_j }(\sigma^2).
	\end{equation*}
We can then define the WPMLE $\hat{\sigma}_{C,p\ell}^2$ of $\sigma_0^2$ as the solution of
	\begin{equation}\label{eq:plcn} \nonumber
		p\ell_{C,n}(\hat{\sigma}_{C,p\ell}^2)=\underset{\sigma^2 \in (0 , \infty)}{\rm min}\,\,p\ell_{C,n} (\sigma^2).
	\end{equation}
	
	The next theorem gives the general inconsistency result for the WPMLE. It is proved at the end of \ref{appendix:inconsistency}.

\begin{theorem} \label{thm:inconsistency:theta:fixed}
Assume that the above assumption holds.
	Then $\hat{\sigma}_{C,p\ell}^2$ does not converge to $\sigma_0^2$ in probability as $n \to \infty$. 
\end{theorem}

Because of its generality, Theorem \ref{thm:inconsistency:theta:fixed} reinforces the warning given in (i) and (iii) of Theorem \ref{thm:consistency:inconsistency:pl} against the WPMLE under fixed-domain asymptotics. Indeed, the estimation of a single variance parameter $\sigma_0^2$, with a known correlation function $C$, is arguably one of the most favorable one for covariance parameter estimation. For instance it is well-known that the MLE of $\sigma_0^2$ is $n^{1/2}$-consistent in the setting of Theorem \ref{thm:inconsistency:theta:fixed} (provided the matrix $[C(\bx_i-\bx_j)]_{i,j \in \{1 , \ldots,n \}}$ is invertible for all $n$). Thus, it would not be a surprise if the WPMLE was shown to be inconsistent under fixed-domain asymptotics, in more general cases than in Theorem \ref{thm:inconsistency:theta:fixed}, where several covariance parameters have to be estimated, for instance a variance and $d$ spatial correlation lengths \cite{roustant2012dicekriging}.

In the next theorem, we show that, for some specific structures of the weights and of the observation points, the WPMLE of $\sigma_0^2$  converges to a non-constant random variable (that is given explicitly in the proof of this theorem).

\begin{theorem} \label{thm:inconsistency:theta:fixed:2}
Assume that the above assumption holds.
	Let $g: [-1,1]^d \to \mathbb{R}^+ $ be a continuous non-zero function satisfying $g(\bx) = g(-\bx)$ for $\bx \in [0,1]^d$. Assume that $w_{i,j} = g(\bx_j - \bx_i)$ for $n \in \mathbb{N}$, $i,j \in \{1,\ldots,n \}$.
	Assume also that there exists $\delta  \in (0,1]$ and $a \in (0,\infty)$ so that $1-C(\bt)= a ||\bt||^\delta + o( ||\bt||^\delta)$ as $\bt \to \bzero$. 
	Let the observation points $(\bx_i)_{i  \in  \mathbb{N}}$ be random, i.i.d. with bounded probability distribution function $f$ on $[0,1]^d$, with $(\bx_i)_{i  \in  \mathbb{N}}$ independent of $Y$. 
	Then there exists a non-negative random variable $\hat{\sigma}_{\infty}^2$ with non-zero variance so that $\mathrm{E}( \hat{\sigma}_{\infty}^2 ) = \sigma_0^2$ and so that, almost-surely, $\hat{\sigma}_{C,p\ell}^2 \to_{n \to \infty} \hat{\sigma}_{\infty}^2$.
\end{theorem}

We remark that the WPMLE of $\sigma_0^2$ is thus an example of an estimator that is consistent under increasing-domain asymptotics and that is inconsistent and converges to a non-degenerate distribution under fixed-domain asymptotics. This type of behavior is also analyzed in \cite{ZhaZim2005}.

\begin{proof}[{\bf Proof of Theorem \ref{thm:inconsistency:theta:fixed}}]

We can show that
{\small
\begin{align*} 
	p\ell_{C,n}
	& =
	\ln(\sigma^2)
	\left(2 \sum_{i=1}^{n-1} \sum_{j=i+1}^n
	w_{i,j}
	\right)
	+	\sum_{i=1}^{n-1} \sum_{j=i+1}^n
	w_{i,j}
	\ln (1-C(\bx_j - \bx_j)^2)+
	\frac{1}{\sigma^2} 
	\sum_{i=1}^{n-1} \sum_{j=i+1}^n
	w_{i,j}
	\frac{Y(\bx_i)^2+Y(\bx_j)^2-2 C(\bx_j - \bx_j) Y(\bx_i)Y(\bx_j)}{1-C(\bx_j - \bx_j)^2} .
\end{align*}}
Hence, it follows that
	\begin{eqnarray} \label{eq:explicit:hat:sigma:square}
	\hat{\sigma}_{C,p\ell}^2
&= &
	\frac{
		1
	}{
		2 \sum_{i=1}^{n-1} \sum_{j=i+1}^n
		w_{i,j}
	}
	\sum_{i=1}^{n-1} \sum_{j=i+1}^n
	w_{i,j}
	\frac{Y(\bx_i)^2+Y(\bx_j)^2-2 C(\bx_j - \bx_j)Y(\bx_i)Y(\bx_j)}{1-C(\bx_j - \bx_j)^2}
 \\ 
	&= &
	\frac{
		1
	}{
		\sum_{i=1}^{n-1} \sum_{j=i+1}^n
		w_{i,j}
	}
	\sum_{i=1}^{n-1} \sum_{j=i+1}^n
	w_{i,j}
	\frac{\left\{ Y(\bx_i)  - Y(\bx_j) \right\}^2 }{2(1-C(\bx_i-\bx_j)^2)} 
	\notag 
	\\
	&  &  +
	\frac{
		1
	}{
		\sum_{i=1}^{n-1} \sum_{j=i+1}^n
		w_{i,j}
	}
	\sum_{i=1}^{n-1} \sum_{j=i+1}^n
	w_{i,j}
	\frac{2-2C(\bx_j - \bx_j) }{2(1-C(\bx_j - \bx_j)^2)}
	Y(\bx_i) Y(\bx_j). \notag
\end{eqnarray}

Let $r(\bt) = 1-C(\bt)$. Then $r(\bt) = o(1)$ as $\bt \to \bzero$. Hence
\begin{equation} \label{eq:DL:general:case}
	\Delta
	:=
	\rm inf_{\bt \in [-1,1]^d} 
	\frac{2-2C(\bt) }{2(1-C(\bt)^2)}
	= 
	\rm inf_{\bt \in [-1,1]^d} 
	\frac{2r(\bt) }{2( 2r(\bt) -r(\bt)^2) }
	>0,
\end{equation}
where we have used the continuity of $C$ and the fact that $C(\bt)<1$ for $\bt \in [-1,1]^d \backslash \{\bzero\}$.
Hence we have, if $\rm inf_{\bt \in [0,1]^d} Y(\bt) \geq 0$,
\begin{equation} \label{eq:for:inconsistency:lower:bound}
	\hat{\sigma}_{C,p\ell}^2
	\geq
	\Delta
	\rm inf_{\bt \in [0,1]^d} Y(\bt)^2 .
\end{equation}

From Lemma A.3 in \cite{lopera17finite}, we have
\[
\rm inf_{\bt \in [0,1]} Y(\bt)^2
\geq 
\frac{ 2\sigma_0^2 }{ \Delta} 
\]
with probability $\Pr>0$ where $\Pr$ does not depend on $n$. Hence, from \eqref{eq:for:inconsistency:lower:bound}, $\hat{\sigma}_{C,p\ell}^2$ does not converge to $\sigma_0^2$ in probability.
\end{proof}

\begin{proof}[{\bf Proof of Theorem \ref{thm:inconsistency:theta:fixed:2}}]
From \eqref{eq:explicit:hat:sigma:square}, we have
{\normalsize
\begin{align*}
	\hat{\sigma}_{C,p\ell}^2
	= &
	\frac{
		1
	}{
		\sum_{i,j=1}^n
		g( \bx_j - \bx_i )
	}
	\sum_{i,j=1}^n
	g( \bx_j - \bx_i )
	\frac{\left\{ Y(\bx_i)  - Y(\bx_j) \right\}^2 }{2(1-C(\bx_i-\bx_j)^2)} 
	\notag 
	\\
	&  +
	\frac{
		1
	}{
		\sum_{i,j=1}^n
		g( \bx_j - \bx_i )
	}
	\sum_{i,j=1}^n
	g( \bx_j - \bx_i )
	\frac{2-2C(\bx_i-\bx_j)}{2(1-C(\bx_i-\bx_j)^2)}
	Y(\bx_i) Y(\bx_j) \\
	= &
	\frac{
		1
	}{
		\frac{1}{n^2} \sum_{i,j=1}^n
		g( \bx_j - \bx_i )
	}
	\frac{1}{n^2}
	\sum_{i,j=1}^n
	F(\bx_i,\bx_j) +
	\frac{ 
		1
	}{
		\frac{1}{n^2}\sum_{i,j=1}^n
		g( \bx_j - \bx_i )
	}
	\frac{1}{n^2}
	\sum_{i,j=1}^n
	G(\bx_i,\bx_j),
\end{align*}}
by defining $F(\bx_i,\bx_j)$ and $G(\bx_i,\bx_j)$ appropriately.
In the above display, $F$ and $G$ are random functions from $[0,1]^d \times [0,1]^d$ to $\mathbb{R}$. From \eqref{eq:DL:general:case}, and since $Y$ has continuous realizations almost surely \cite{strait1966sample} and $g$ is continuous, $G$ is a bounded continuous function almost surely. Also, $F$ is continuous on $[-1,1]^d \backslash \{ \bzero \}$ almost surely. 
From \cite{strait1966sample}, for $\alpha < \delta$, almost surely, there exists $A < \infty$ for which, for any $\bx,\by \in [0,1]^d$, $ | Y(\bx) - Y(\by) | \leq A ||\bx - \by||^{\alpha/2}$. Hence, almost surely, $F( \bx,\by ) \leq B ||\bx - \by||^{ \alpha - \delta }$ with $B < \infty$. The function $(\bx,\by) \to \|\bx - \by\|^{ \alpha -  \delta}$ has a finite integral on $[-1,1]^d$ for $\alpha$ close enough to $\delta$.
Hence, almost surely,
\[
\int_{[0,1]^d}
\int_{[0,1]^d}
F(\bx,\by) f(\bx) f(\by) d\bx d\by
< \infty.
\]

Hence, for almost all the realization functions of $Y$ (with respect to the distribution of $Y$), we can use Theorem A in Section 5.4 of \cite{serfling2009approximation} to obtain that
\begin{align*}
	\hat{\sigma}_{C,p\ell}^2
	\to_{n \to \infty} 
	\frac{
		1
	}{
		\int_{[0,1]^d}
\int_{[0,1]^d}  g(\by-\bx) f(\bx) f(\by) d\bx d\by
	}
	\int_{[0,1]^d}
\int_{[0,1]^d}
	g(\by - \bx)
	\frac{\left\{ Y(\bx)  - Y(\by) \right\}^2 }{2(1-C(\bx - \by)^2)} 
	f(\bx) f(\by) d\bx d\by \\
	+
	\frac{
		1
	}{
		\int_{[0,1]^d}
\int_{[0,1]^d}  g(\by-\bx) f(\bx) f(\by) d\bx d\by
	}
	\int_{[0,1]^d}
\int_{[0,1]^d}	
	g(\by - \bx)
	\frac{2-2C(\bx-\by) }{2(1-C(\bx - \by)^2)}
	Y(\bx) Y(\by)
	f(\bx) f(\by) d\bx d\by. 
\end{align*}
The above limit is a random variable that we write $\hat{\sigma}_{\infty}^2$ and that is almost surely finite (the integrals are defined in the $L^2$ sense, see \cite{Ste1999}). One can show that the mean value of $\hat{\sigma}_{\infty}^2$ is $\sigma_0^2$. Finally, the variance of $\hat{\sigma}_{\infty}^2$ is non-zero because of Theorem \ref{thm:inconsistency:theta:fixed}.
\end{proof}

%To ensure accuracy, get them from MathSciNet whenever possible. Typeset them with BibTeX using JMVA's style file, \texttt{myjmva.bst}.

\bibliographystyle{abbrv}
\bibliography{clasymf}
%\bibliography{}

%\begin{thebibliography}
%\end{thebibliography}

\end{document}